\documentclass{article}
\usepackage[a4paper, total={5.1in, 8in}]{geometry}
\usepackage{graphicx} 
\usepackage{imakeidx}
\makeindex

\usepackage{amsmath}
\usepackage{amssymb}
\usepackage{amsthm}
\usepackage{bussproofs}
\usepackage{makecell}
\usepackage{multicol}
\usepackage{bbm}
\usepackage{xcolor}
\usepackage{graphicx}
\usepackage{enumitem}
\usepackage{booktabs}
\usepackage{tikz}

\usepackage{hyperref}
\usepackage{cleveref}

\newtheorem{theorem}{Theorem}[section]
\newtheorem{definition}[theorem]{Definition}%
\newtheorem{proposition}[theorem]{Proposition}%
\newtheorem{lemma}[theorem]{Lemma}%
\newtheorem{fact}[theorem]{Fact}%
\newtheorem{example}[theorem]{Example}%

\DeclareFontFamily{U}{mathb}{}
\DeclareFontShape{U}{mathb}{m}{n}{<-5.5> mathb5 <5.5-6.5> mathb6 
<6.5-7.5> mathb7 <7.5-8.5> mathb8 <8.5-9.5> mathb9 <9.5-11> mathb10 
<11-> mathb12}{}
\DeclareFontSubstitution{U}{mathb}{m}{n}
\DeclareRobustCommand{\blackdiamond}{\mathbin{\text{\usefont{U}{mathb}{m}{n}\symbol{"0C}}}}

\usepackage{caption} 

\usepackage{hypcap}

\def\vdashv{\dashv\vdash}
\newcommand{\fullsubseteq}{\!\mathrel{\rlap{\raisebox{0.245ex}{\hspace{0.45em}$\bullet$}}\subseteq}}

\def\sempty{\langle\rangle}

\newcommand{\uglor}%
 {\mathop{%
    \raisebox{0.25ex}{\scalebox{0.65}{\rotatebox{4}{$\setminus$}}}%
    \hspace{-2.6pt}%
    \scalebox{1}{$\lor$}%
    \hspace{-4.3pt}%
    \raisebox{0.66ex}{\scalebox{0.65}{\rotatebox{-44}{$\setminus$}}}%
  }}

\newcommand{\glor}{%
  \mathop{%
    \raisebox{0.25ex}{\scalebox{0.65}{\rotatebox{4}{$\setminus$}}}%
    \hspace{-2.6pt}%
    \raisebox{0pt}{%
      \scalebox{1}{%
        \ooalign{%
          $\lor\!\!$
        }%
      }%
    }%
  }%
}

\newcommand{\Glor}{\mathop{
\scalebox{1.5}
{\raisebox{-.27ex}{\scalebox{0.95}{\rotatebox{4}{$\setminus$}}\hspace{-4pt}\scalebox{1.5}{\raisebox{-.2ex}{$\lor$}}}}}}

\newcommand{\PLprim}{PL\ensuremath{(\subseteq_0)}}

\newcommand{\PLglor}{PL^{\glor}}

\usepackage{enumitem}

\usepackage{oplotsymbl}
\let\singlemight=\trianglepbdot
\let\might=\trianglepb

\title{Capturing dual team properties with inclusion atoms}
\author{Matilda Häggblom}
\date{University of Helsinki}

\begin{document}

\maketitle

\noindent \textbf{Abstract. }
We introduce propositional team-based logics expressively complete for (quasi) downward and (quasi) upward closed properties in a syntactically dual way, by using variants of the inclusion atom. In particular, the variants of the primitive inclusion atoms used in the (quasi) upward closed setting have equivalent formulas using variants of the might modality. The duality is visible in the logics' normal forms, mirroring the duality between the (quasi) upward and downward closed settings, where the quasi variants take special care of the empty and full team. Furthermore, we defined sound and complete natural deduction systems for each logic.

\section{Introduction}

We ask whether we can define propositional team-based logics in a \emph{dual} way, reflecting the duality between downward and upward closed properties, with (and without) the empty/full team. We answer affirmatively by introducing such logics using variants of the inclusion atom, producing a surprisingly symmetric picture seen in the four logics’ normal forms. We also introduce a sound and complete natural deduction system for each logic.  Notably, the atoms in the logics for (quasi) upward closed properties have close connections to the might modalities in the literature.

We show that each of the four logics is expressively complete for all team properties of the relevant closure properties, and define sound and complete proof systems. All propositional team logics are compact, using arguments from inquisitive logic in \cite{ciardelli2009}, presented in \cite{YANG2017} with the team semantics notation. Hence, we obtain strong compactness theorems for each logic.

\bigbreak

This is a working paper.

\section{Upward and downward closed logics}

We recall basic definitions about teams, team properties and propositional team logic. We then introduce four variants of upward and downward closed propositional team logics and provide each with an expressive completeness result.  

Fix a finite set of propositional symbols $\mathbb{P}$. A set of valuations $v:\mathbb{P}\rightarrow\{0,1\}$ is a team. We say here that $\mathcal{C}$ is a team property if it is a nonempty collection of teams over $\mathbb{P}$, i.e., $\mathcal{C}\subseteq P(2^{\mathbb{P}})$. 

Two interesting teams are the empty team $\emptyset$ and the full team $\mathbb{F}:=2^{\mathbb{P}}$, i.e.,  the maximal team over $\mathbb{P}$. We consider the empty/full team property of any collection $\mathcal{C}$ of teams.

\begin{enumerate}[label=\textbf{-}]
    \item $\mathcal{C}$ has the empty team property\index{empty team property}  if $\emptyset\in\mathcal{C}$.

    \item $\mathcal{C}$ has the full team property\index{full team property} if $\mathbb{F}\in\mathcal{C}$.
\end{enumerate}

The four closure properties central to the results of this paper are listed next, where we note the special roles of the full and empty teams in the quasi variants.

\begin{enumerate}[label=\textbf{-}]
    \item $\mathcal{C}$ is downward closed\index{downward closed}  if for all $T\in\mathcal{C}$ and $S\subseteq T$, $S\in\mathcal{C}$.

    \item $\mathcal{C}$ is \emph{quasi} downward closed\index{quasi downward closed}  if $\mathbb{F}\in\mathcal{C}$ and $\mathcal{C}\setminus\{\mathbb{F}\}$ is downward closed.

    \item $\mathcal{C}$ is upward closed\index{upward closed} if for all $T\in\mathcal{C}$ and $S\supseteq T$, $S\in\mathcal{C}$.

    \item $\mathcal{C}$ is \emph{quasi} upward closed\index{quasi upward closed} if $\emptyset\in\mathcal{C}$ and $\mathcal{C}\setminus\{\mathsf{\emptyset}\}$ is upward closed.
\end{enumerate}

Note that if we have a downward closed property that contains the full team, it would contain all teams and thus represent the trivial team property. The symmetrical situation appears for an upward closed team property containing the empty team. In contrast, the quasi properties avoid this trivialization. 

For a team-based logic $\mathcal{L}$, each formula $\phi \in\mathcal{L}$ (with propositional symbols from $\mathbb{P}$) defines a team property $\lVert \phi\rVert:=\{t\subseteq\mathbb{F}\mid T\models\alpha\}$. For a set of formulas $\Gamma$, we write $\Gamma\models\psi$ if $\lVert \psi\rVert\supseteq \bigcap_{\phi\in\Gamma}\lVert \phi\rVert$. If $\Gamma$ is a singleton $\{\phi\}$, we drop the brackets and simply write $\phi\models\psi$. We say that two formulas $\phi$ and $\psi$ are semantically equivalent if $\lVert \psi\rVert= \lVert \phi\rVert$ and denote this by $\phi\equiv\psi$.

We say that a logic has a specific (closure) property if all formulas $\phi$ in the logic are such that $\lVert\phi\rVert$ has that property. If, additionally, for any collection of teams $\mathcal{C}$ with the desired property, there is some $\phi$ in the logic such that $\mathcal{C}=\lVert \phi\rVert$, we say that the logic is expressively complete\index{expressive completeness} for such team properties.  
For instance, $\mathcal{L}$ is expressively complete for all downward closure team properties containing the empty team if and only if  
$$\{\lVert \phi\rVert\mid \phi\in\mathcal{L}\}=\{\mathcal{C}\mid \emptyset\in\mathcal{C}\text{ and $\mathcal{C}$ is downward closed} \}.$$

This is the case for classical propositional team logic ($PL$) extended with the global disjunction $\glor$, making $\PLglor$ expressively complete for all downward closed team properties with the empty team \cite{Yang16}. 
The syntax for $\PLglor$\index{$\PLglor$} is given by the grammar:
\begin{equation*}
   \phi ::= \bot\mid p\mid\neg p\mid(\phi\land \phi)\mid(\phi\lor\phi)\mid(\phi\glor\phi),
\end{equation*}
where $p\in\mathbb{P}$.  We obtain the logic $PL$\index{$PL$} by restricting to the $\glor$-free fragment. 
We recall the semantic clauses for $\PLglor$. 
\begin{align*}
T\models\top  \quad \text{iff}\quad& \text{always}.\\
T\models\bot  \quad \text{iff}\quad&T=\emptyset.\\
T\models p  \quad \text{iff}\quad& v(p)=1 \text{ for all }v\in T. \\
T\models\neg p  \quad \text{iff}\quad& v(p)=0 \text{ for all }v\in T.  \\
T\models\phi\land\psi  \quad \text{iff}\quad& T\models\phi \text{ and } T\models\psi. \\
 T\models\phi\lor\psi     \quad \text{iff}\quad& T_1\models\phi \text{ and } T_2\models\psi \text{ for some } T_1,T_2\subseteq T \text{ such that }  T_1\cup T_2=T.\\
T\models\phi\glor\psi  \quad \text{iff}\quad& T\models\phi \text{ or } T\models\psi. 
 \end{align*}

 In particular, it is easy to see that the team property $\lVert p \rVert$ is downward closed and not upward closed. Since our goal is to define both upward and downward closed logics in a dual way, we need some alternative atomic formulas. For this purpose, we use variants of inclusion atoms with constants as our main atoms.

 Propositional inclusion atoms with constants are of the form $\mathsf{a\subseteq\mathsf{b}}$, where $\mathsf{a}$ and $\mathsf{b}$ are sequences of the same length consisting of propositional symbols and constants $\top$ and $\bot$. Let us extend the domain of the valuations to include the constants, such that valuations are of the form $v:\mathbb{P}\cup\{\top,\bot\}\rightarrow\{0,1\}$ and $v(\top)=1$ and $v(\bot)=0$ always hold. We also stipulate that $\bigwedge\emptyset=\top$ and $\bigvee\emptyset=\bot$.
 
 For a valuation $v$ and a sequence of propositional symbols and constants $\mathsf{a}=a_1\dots a_n$, we write $v(\mathsf{a})$ as shorthand for the tuple $ v(a_1)\dots  v(a_n)$. Now
 $$T\models \mathsf{a}\subseteq \mathsf{b}  \quad \text{iff}\quad \text{for all $v\in T$, there is some $v'\in T$ such that $v(\mathsf{a})=v'(\mathsf{b})$}   .$$ 

 Inclusion atoms are neither downward nor upward closed, but they have the empty team property and are union closed, that is, if $T_i\models\mathsf{a}\subseteq \mathsf{b}$ for all $i\in I\neq \emptyset$, then $\bigcup_{i\in I}T_i\models\mathsf{a}\subseteq \mathsf{b}$.

We now introduce two novel variants of the inclusion atom, the nonempty inclusion atom $\mathsf{a}\subseteqq \mathsf{b}$ and the full inclusion atom $\mathsf{a}\fullsubseteq \mathsf{b}$, with the following semantic clauses.
 \begin{align*}
    T\models\mathsf{a}\subseteqq \mathsf{b}\quad \text{iff}\quad& t\neq\emptyset \text{ and } T\models \mathsf{a}\subseteq \mathsf{b}.\\
T\models\mathsf{a}\fullsubseteq \mathsf{b}\quad \text{iff}\quad& T=\mathbb{F} \text{ or }T\models \mathsf{a}\subseteq \mathsf{b}. 
\end{align*}

To obtain inclusion atoms with the desired (quasi) upward closed properties, we restrict the syntax of the sequences $\mathsf{a}$ and $\mathsf{b}$. Let $\mathsf{p}$ be a sequence of propositional symbols from $\mathbb{P}$ and let $\mathsf{x}$ be a sequence of constants $\top,\bot$ such that $|\mathsf{p}|=|\mathsf{x}|$. Primitive inclusion atoms from \cite{yang_propositional_2022} are of the form $\mathsf{x}\subseteq \mathsf{p}$. For a nonempty team $T$, $T\models \mathsf{x}\subseteq \mathsf{p}$ if and only if there is a valuation $v\in T$ such that $v(\mathsf{p})=v(\mathsf{x})$. Primitive inclusion atoms are quasi upward closed, since if $T\models \mathsf{x}\subseteqq \mathsf{p}$ and $S\supseteq T$, then there is a valuation $v\in T$ such that $v(\mathsf{p})=v(\mathsf{x})$, hence also $v\in S$ and $S\models\mathsf{x}\subseteq \mathsf{p}$ follows. 

By similar observations, the nonempty primitive inclusion atom $\mathsf{x}\subseteqq\mathsf{p}$ is upward closed, since now for all teams $T$, $T\models \mathsf{x}\subseteq \mathsf{p}$ if and only if there is a valuation $v\in T$ such that $v(\mathsf{p})=v(\mathsf{x})$. Since we aim to obtain logics expressively complete for nonempty team properties, we further demand that the sequence of propositional symbols $\mathsf{p}$ in a nonempty inclusion atom $\mathsf{x}\subseteqq\mathsf{p}$ does not have any repeated propositional symbols. Otherwise, we would have $\top\bot\subseteqq p_1p_1$ which is satisfied by no team. For the sake of uniformity, we chose to extend this syntactical restriction also to the primitive inclusion atoms $\mathsf{x}\subseteq\mathsf{p}$ that will be included in the logic $\mathcal{L}_{qu}$. 

To obtain atoms suitable for the downward closed setting, we consider inclusion atoms of the form $\mathsf{p}\subseteq\mathsf{x}$, dual of the primitive inclusion atom.  In particular, for $\mathsf{p}=p_1\dots p_n$ and $\mathsf{x}=x_1\dots x_n$, we have the semantic equivalence $\mathsf{p}\subseteq\mathsf{x}\equiv \mathsf{p}^\mathsf{x}$, where  $\mathsf{p}^\mathsf{x}=p_1^{\mathsf{x}_1}\land\dots \land p_n^{\mathsf{x}_n}$ with $p_i^{\top}=p$ and $p_i^{\bot}=\neg p$. Since a conjunction of literals defines a downward closed team property, so does the dual primitive inclusion atom.
Similarly, full inclusion atoms of the form $\mathsf{p}\fullsubseteq\mathsf{x}$ define quasi downward closed team properties.

To summarize, we obtain the desired property of each inclusion atom: primitive inclusion atoms  $\mathsf{x}\subseteq \mathsf{p}$ are 
quasi upward closed, nonempty primitive inclusion atoms  $\mathsf{x}\subseteqq \mathsf{p}$ are 
upward closed, dual full primitive inclusion atoms $\mathsf{p}\fullsubseteq \mathsf{x}$ are quasi downward closed, and finally, dual primitive inclusion atoms $\mathsf{p}\subseteq \mathsf{x}$ are downward closed.

With the main atoms defined, we are now ready to give the syntax of our four logics. 

\begin{definition}
    The grammars of the logics $\mathcal{L}_{qu}$, $\mathcal{L}_{u}$, $\mathcal{L}_{qd}$ and $\mathcal{L}_{d}$ are as follows.

\begin{enumerate}
    \item[$(\mathcal{L}_{qu})$] \centering \quad $\phi::= \bot \mid \mathsf{x}\subseteq \mathsf{p} \mid (\phi\land\phi)\mid (\phi\glor\phi).$ \hspace{1.2cm}

    \item[$(\mathcal{L}_{u})$] \centering \quad $\phi::= \top \mid \mathsf{x}\subseteqq \mathsf{p} \mid (\phi\land\phi)\mid (\phi\glor\phi)$.  \hspace{1.2cm}

     \item[$(\mathcal{L}_{qd})$] \centering \quad $\phi::= \bullet\mid\mathsf{p}\fullsubseteq\mathsf{x} \mid (\phi\land\phi) \mid (\phi\lor\phi)\mid (\phi\glor\phi)$.

      \item[$(\mathcal{L}_d)$] \centering \quad $\phi::= \bot\mid\mathsf{p}\subseteq\mathsf{x} \mid (\phi\land\phi) \mid (\phi\lor\phi)\mid (\phi\glor\phi)$.
\end{enumerate}
\end{definition}

The semantics of the constants, atoms and connectives are familiar, and we add the semantic clause for the \emph{full atom} $\bullet$, for which
 \begin{align*}
    T\models\bullet\quad \text{iff}\quad& T=\mathbb{F}. 
\end{align*}

Next, we show that the logics have the desired closure properties. 

\begin{proposition}
The logic $\mathcal{L}_{qu}$ is quasi upward closed, $\mathcal{L}_u$ is upward closed, $\mathcal{L}_{qd}$ is quasi downward closed, and $\mathcal{L}_d$ is downward closed.
\end{proposition}

\begin{proof}
Proving the claim for each logic by induction on the complexity of the formulas is straightforward; here, we include the induction cases for $\land$ and $\glor$ for the logic $\mathcal{L}_{qu}$.

Let $\phi$ and $\psi$ be quasi upward closed formulas and assume that $S\supseteq T$. By the empty team property of the formulas, clearly $\emptyset\models \phi\land\psi$. Now suppose that $T$ is nonempty and $T\models \phi\land\psi$. Then $T\models \phi$ and $T\models \psi$, so by quasi upward closure of the formulas and the assumption that $T$ is nonempty, $S\models \phi$ and $S\models \psi$, hence $S\models \phi\land \psi$. 

Consider now global disjunctive formulas $\phi\glor \psi$. By assumption $\phi$ has the empty team property, hence  $\emptyset\models\phi$ and $\emptyset\models\phi\glor\psi$ follows. Now, consider a nonempty team $T$ such that $T\models \phi\land\psi$. Then  $T\models \phi$ or $T\models \psi$, so by quasi upward closure of the formulas, $S\models \phi$ or $S\models \psi$, hence $S\models \phi\glor \psi$ as desired. 
\end{proof}

A unusual feature of the quasi upward closed setting is that the global disjunction $\glor$ and split-disjunction $\lor$ coincide.

\begin{proposition}\label[proposition]{qu glor equiv lor}
    Let $\phi$ and $\psi$ have the empty team property and further assume that $\psi$ is quasi upward closed, then $\phi\lor\psi\equiv \phi\glor\psi$. 
\end{proposition}
\begin{proof}
Since both $\phi$ and $\psi$ have the empty team property, the semantic equivalence holds over the empty team. Suppose now that $T\neq\emptyset$. 

First, let $T\models\phi\lor\psi$. Then there are subteams $T_1,T_2\subseteq T$ such that $T=T_1\cup T_2$, $T_1\models \phi$ and $T_2\models\phi$. If $T_2=\emptyset$, then $T_1=T$, hence $T\models\phi$ and $T\models\phi\glor\psi$ follows. Otherwise, $T_2\neq\emptyset$ and thus by quasi upward closure of $\psi$, $T\models\psi$ from which $T\models\phi\glor\psi$ follows.

The other direction follows by the empty team property of the formulas. Suppose that $T\models \phi\glor\psi$. If $T\models \phi$, then $T=T\cup\emptyset\models \phi\lor\psi$. Similarly, if  $T\models \psi$, then $T=\emptyset\cup T\models \phi\lor\psi$. 
\end{proof}

We also consider the strict disjunction\index{strict disjunction} $\lor_s$ with the following semantic clause.  \begin{align*}
    T\models\phi\lor_s\psi\quad \text{iff}\quad& \text{there are disjoint teams $T_1$ and $T_2$ such that $T_1\cup T_2=T$,} \\
    &\text{$T_1\models\phi$ and $T_2\models\psi$}. 
\end{align*}
It is easy to see that in the quasi upward closed setting, $\phi\lor_s\psi\equiv\phi\lor\psi\equiv\phi\glor\psi$, making this setting resilient to changes in the semantic clause of its disjunction.

In the upward closed setting, without the empty team property, $\phi\lor\psi\models \phi\glor\psi$ still holds, but not the other direction: $\phi\glor\psi\not\models \phi\lor\psi$ exemplified by $\emptyset\models\top\glor \top\subseteqq p$ but $\emptyset\not\models \top\lor \top\subseteqq p$. Instead, we observe a different semantic equivalence: that between split-disjunction and conjunction.

\begin{proposition}
 Let $\phi$ and $\psi$ be upward closed formulas. Then $\phi\lor\psi\equiv \phi\land\psi$. 
    \end{proposition}
\begin{proof}
The right-to-left direction is trivial, so we show $\phi\lor\psi\models \phi\land\psi$. If $T\models\phi\lor\psi$, then there are subteams $T_1$ and $T_2$ of $T$ such that $T_1\cup T_2=T$, $T_1\models\phi$ and $T_2\models\psi$. By upward closure of the formulas, $T_1\cup T_2\models\phi$ and $T_1\cup T_2\models\psi$, hence $T_1\cup T_2=T\models\phi\land\psi$. 
\end{proof}

We end with a known fact, the split-disjunction and conjunction distribute over the global disjunction unconditionally.

\begin{fact}\label[fact]{fact glor distr}
    Let $\phi,\psi$ and $\chi$ be any formulas. Then $\psi\land(\chi\glor\theta)\equiv (\psi\land\chi)\glor(\phi\land\theta)$ and   $\psi\lor(\chi\glor\theta)\equiv (\psi\lor\chi)\glor(\phi\lor\theta)$.
\end{fact}

In the following four sections, we study each of the four logics in detail by showing their exact expressive power as well as introducing their respective sound and complete natural deduction proof systems.

To obtain the expressive completeness results, we define a normal form for each logic with which each desired property can be expressed. The normal forms also play an important role in proving the completeness theorem for the proof systems we introduce. For the (quasi) upward closed logics, the system is obtained by adapting the one for propositional inclusion logic \cite{yang_propositional_2022}. For the (quasi) downward closed logics, we adapt the system for $\PLglor$ from \cite{Yang16}. 

\subsection{Quasi upward closed logic}

We show that the logic $\mathcal{L}_{qu}$ is expressively complete for all quasi upward closed team properties with the empty team and provide a complete axiomatization. 

Recall the syntax of  $\mathcal{L}_{qu}$. \[\phi::= \bot \mid \mathsf{x}\subseteq \mathsf{p} \mid (\phi\land\phi)\mid (\phi\glor\phi).\] 

Define $\top:=\sempty\subseteq\sempty$, where $\sempty$ denotes the empty sequence.


We aim to show that for any quasi upward closed team property $\mathcal{C}$, there is a formula in $\mathcal{L}_{qu}$ that defines it. 
We obtain this result by using the logic's \emph{normal form}.  

Let the sequence $\mathsf{p}$ contain all propositional symbols in $\mathbb{P}$ and recall that we assume $\mathbb{P}$ to be finite.  For teams $T$, define $$\psi^{\prime}_T:=\bigwedge_{v\in T} \mathsf{x}^v\subseteq \mathsf{p},$$ with $\mathsf{x}^v={v(1)}\dots {v(n)}$, where ${v(i)}=\top$ if $v(p_i)=1$ and $\bot$ otherwise. The formula $\psi^{\prime}_T$ essentially appears as a subformula in the normal form for $\PLprim$ in \cite{yang_propositional_2022}. 
Note that since we define $\mathsf{p}$ to include all propositional symbols in $\mathbb{P}$, there is one unique valuation $v'$ over $\mathbb{P}$ for which $\{v'\}\models \mathsf{x}^v\subseteq \mathsf{p}$, hence $v'=v$. Furthermore, for nonempty teams $T$,  $T\models \mathsf{x}^v\subseteq \mathsf{p}$ if and only if $v\in T$.

We examine some basic properties of the $\psi^{\prime}_T$-formulas in the next lemma, where \cref{wuc lm 1 item} is due to \cite{yang_propositional_2022}. 

\begin{lemma} \label[lemma]{wuc lm 1}
Let $\mathcal{C}$ and $\mathcal{D}$ be nonempty team properties not containing the empty team. 
\begin{enumerate}[label =$($\roman*$)$]

      
    \item  $S\models \psi^{\prime}_T$ iff $S=\emptyset$ or $S\supseteq T$. \label{wuc lm 1 item}

 \item $\Glor_{S\in\mathcal{D}}\psi^{\prime}_S\models \Glor_{T\in\mathcal{C}}\psi^{\prime}_T$ iff for each $S\in\mathcal{D}$ there is some $T\in\mathcal{C}$ such that $S\supseteq T$. \label{wuc lm 2 item}
\end{enumerate} 
\end{lemma}
\begin{proof}
\begin{enumerate}
    \item[\ref{wuc lm 1 item}] 
If $S=\emptyset$, the claim follows by the empty team property. If $T=\emptyset$, both sides of the equivalence are trivially satisfied, since $\bigwedge\emptyset=\top$. 

   Suppose $S\neq\emptyset$ and let $S\models \psi^{\prime}_T$, where  $\psi^{\prime}_T=\bigwedge_{v\in T} \mathsf{x}^v\subseteq \mathsf{p}$. Then for each valuation $v\in T$,  $S\models \mathsf{x}^v\subseteq \mathsf{p}$, which implies that $v\in S$, hence $S\supseteq T$. 
For the other direction, assuming that $S\supseteq T$, we clearly have that $T\models \psi^{\prime}_T$, hence $S\models \psi^{\prime}_T$ follows by quasi upward closure of the formula.

 \item[\ref{wuc lm 2 item}] Let $\mathcal{C}$ and $\mathcal{D}$ be as stated. Suppose first that $\Glor_{S\in\mathcal{D}}\psi^{\prime}_S\models \Glor_{T\in\mathcal{C}}\psi^{\prime}_T$. Then for each $T\in\mathcal{D}$ there is some $S\in\mathcal{C}$ such that $\psi^{\prime}_S\models\psi^{\prime}_T$. Since $S\models \psi^{\prime}_S$, also $S\models\psi^{\prime}_T$ and $S\supseteq T$ follows by \cref{wuc lm 1 item}. 
 
 For the other direction, let $S'\models \Glor_{S\in\mathcal{D}}\psi^{\prime}_S$. If $S'=\emptyset$, the claim follows by the empty team property, so suppose that $S'$ is nonempty. Now there is some $S\in\mathcal{D}$ such that $S'\models\psi^{\prime}_S$, which by \cref{wuc lm 1 item} entails that $S'\supseteq S$. By assumption, there is some $T\in\mathcal{C}$ such that $S'\supseteq S\supseteq T$. Since $T\models\psi^{\prime}_T$, and $T\neq\emptyset$, it follows by weak upward closure of the formula that $S'\models\psi^{\prime}_T$, hence  $ S'\models\Glor_{T\in\mathcal{C}}\psi^{\prime}_T$.

\end{enumerate}
\end{proof}

The expressive completeness result is now obtained by considering the normal form $$\Psi^{\prime}_\mathcal{C}:=\Glor_{T\in\mathcal{C}}\psi^{\prime}_T,$$ where $\mathcal{C}$ is a nonempty team property without the empty team. We seperately define $\Psi^{\prime}_{\{\emptyset\}}:=\bot$. 

\begin{theorem}
    $\mathcal{L}_{qu}$ is expressively complete for all quasi upward closed team properties $\mathcal{C}$.
\end{theorem} 
\begin{proof}



     Every $\phi\in\mathcal{L}_{qu}$ defines a nonempty quasi upward closed property, so it remains to show that any nonempty quasi upward closed property $\mathcal{C}$ over $\mathbb{P}$ is definable by some formula $\phi\in \mathcal{L}_{qu}$, i.e.,   
$\mathcal{C} = \lVert\phi\rVert$.  

If $\mathcal{C}=\{\emptyset\}$, then the formula $\bot$ clearly captures the team property. So suppose that $\{\emptyset\}\subsetneq\mathcal{C}$. Now 
for any team $S$ over $\mathbb{P}$, we have 
\begin{align*}
  S \models  \Glor_{T \in \mathcal{C}\setminus\{\emptyset\}} \psi^{\prime}_T \iff & S\models \psi^{\prime}_T \text{ for some } T\in \mathcal{C}\setminus\{\emptyset\}  \\
    \iff & S=\emptyset \text{ or } S\supseteq T \text{ for some } T\in \mathcal{C}\setminus\{\emptyset\} \\
  \iff & S \in \mathcal{C},
\end{align*}
where the second equivalence is by  \Cref{wuc lm 1} \cref{wuc lm 2 item}, and the last equivalence is by quasi upward closure of $\mathcal{C}$. Thus $\mathcal{C}=\lVert\Glor_{T \in \mathcal{C}\setminus\{\emptyset\}} \psi^{\prime}_{T}\rVert$, as desired.
\end{proof}




We now present a proof system for the logic $\mathcal{L}_{qu}$ in \Cref{rules:wuc_logic}.

\begin{table}[ht]\centering \small
  \renewcommand*{\arraystretch}{3.8}
\begin{tabular*}{\linewidth}{@{\extracolsep{\fill}}|ccc|}
\hline
     \AxiomC{} 
     \noLine     
     \UnaryInfC{$\bot$}
     \RightLabel{$\bot$E}
\UnaryInfC{$\phi$}
\DisplayProof \hspace{.2cm}

     \AxiomC{} 
     \noLine     
     \UnaryInfC{}
     \RightLabel{$\top$I}
\UnaryInfC{$\sempty\subseteq\sempty$}
\DisplayProof

&
\AxiomC{}
\noLine
\UnaryInfC{$\phi$}
\AxiomC{}
\noLine
\UnaryInfC{$\psi$}
\RightLabel{$\land$I}
\BinaryInfC{$\phi\land\psi$}
\DisplayProof
&
\AxiomC{}
\noLine
\UnaryInfC{$\phi \land \psi$}
\RightLabel{$\land$E}
\UnaryInfC{$\phi$}
\DisplayProof
\hspace{.5cm}

\AxiomC{}
\noLine
\UnaryInfC{$\phi \land \psi$}
\RightLabel{$\land$E}
\UnaryInfC{$\psi$}
\DisplayProof 
 \\ [2ex]

\AxiomC{}
\noLine
\UnaryInfC{$\phi$}
\RightLabel{$\glor$I}
\UnaryInfC{$\phi\glor\psi$}
\DisplayProof
& 
\AxiomC{}
\noLine
\UnaryInfC{$\psi$}
\RightLabel{$\glor$I}
\UnaryInfC{$\phi\glor\psi$}
\DisplayProof &  \AxiomC{}
\noLine
\UnaryInfC{$\phi \glor \psi$}
\noLine
\AxiomC{[$\phi$]}
\noLine
\UnaryInfC{\vdots}
\noLine
\UnaryInfC{$\chi$}
\AxiomC{[$\psi$]}
\noLine
\UnaryInfC{\vdots}
\noLine
\UnaryInfC{$\chi$}
\RightLabel{$\glor$E}
\TrinaryInfC{$\chi$}
\DisplayProof\\ [4ex] 

 \AxiomC{}
     \noLine  
     \UnaryInfC{$\mathsf{x}y\subseteq\mathsf{p}q$}
     \RightLabel{$\subseteq$Proj}
\UnaryInfC{$\mathsf{x}\subseteq\mathsf{p}$}
\DisplayProof
&

 \AxiomC{}
     \noLine  
     \UnaryInfC{$\mathsf{x}\mathsf{y}\mathsf{z}\subseteq\mathsf{u}\mathsf{v}\mathsf{w}$}
     \RightLabel{$\subseteq$Perm (1)}
\UnaryInfC{$\mathsf{x}\mathsf{z}\mathsf{y}\subseteq\mathsf{u}\mathsf{w}\mathsf{v}$} 
 \DisplayProof
& 
\AxiomC{}
 \noLine
 \UnaryInfC{$\mathsf{x}\subseteq\mathsf{p}$}

 \AxiomC{$[\mathsf{x}\top\subseteq\mathsf{p}q]$}
 \noLine
 \UnaryInfC{\vdots}
 \noLine
 \UnaryInfC{$\chi$} 

 \AxiomC{$[\mathsf{x}\bot\subseteq\mathsf{p}q]$}
 \noLine
 \UnaryInfC{\vdots}
 \noLine
 \UnaryInfC{$\chi$} 
 \RightLabel{$\subseteq$Ext}
 \TrinaryInfC{$\chi$}
 \DisplayProof
 \\ [4ex] 
 \hline
 
\multicolumn{3}{|c|}{(1) Provided that $|y|=|v|$ and $|z|=|w|$.}
 \\ 
\hline
\end{tabular*}
 \caption{System for $\mathcal{L}_{qu}$.}
\label{rules:wuc_logic}
\end{table}  

The system is a variant of the one for $\PLprim$ in \cite{yang_propositional_2022} restricted to the quasi upward closed fragment. We include the rules $\bot$E and $\top$I, where we recall that $\sempty\subseteq\sempty$ is a tautology. The rules $\land$I and $\land$E for conjunction and the rules $\glor$I and $\glor$E for the global disjunction are standard. Recall that in the quasi upward closed setting, we could as well replace $\glor$ with $\lor$, and that the obtained $\lor$E rule would be without any added restrictions on the formulas; this is unlike the union closed setting of $\PLprim$, where we must require the undicharged assumptions to be from the downward closed fragment of the logic. The rules we add for the primitive inclusion atom include the standard $\subseteq$Proj and $\subseteq$Perm rules. We can omit the weakening and transitivity rules that would allow for the derivations  $\mathsf{x}y\subseteq\mathsf{p}q\vdash \mathsf{x}yy\subseteq\mathsf{p}qq$, and $\mathsf{a}\subseteq\mathsf{b}, \mathsf{b}\subseteq\mathsf{c}\vdash \mathsf{a}\subseteq\mathsf{c}$, since the shape of primitive inclusion atoms in this logic syntactically prevents such interactions. Lastly, the rule $\subseteq$Ext is novel and captures that if we know that a primitive inclusion atom $\mathsf{x}\subseteq\mathsf{p}$ is satisfied in a nonempty team, then there is a valuation $v$ in the team that satisfies it, and either $v(q)=1$ or $v(q)=0$. If we can derive $\chi$ in both cases, then $\chi$ follows.

\begin{theorem} \label{wuc_soundness}
    The rules in \Cref{rules:wuc_logic}
 are sound for quasi upward closed formulas. \end{theorem}

 \begin{proof}
We show soundness of the rule $\subseteq$Ext. Since $\chi$ has the empty team property, soundness is trivial for the empty team, so suppose that $T\neq \emptyset$. Let $T\models\mathsf{x}\subseteq\mathsf{p}$ and $T\models \Gamma$. Assume further that $\Gamma, \mathsf{x}\top\subseteq\mathsf{p}q\models \chi$ and $\Gamma, \mathsf{x}\bot\subseteq\mathsf{p}q\models \chi$. Recall that we assume that $q$ does not appear in $\mathsf{p}$. Since $T\models\mathsf{x}\subseteq\mathsf{p}$, there is some valuation $v$ in $T$ such that $\{v\}\models \mathsf{x}\subseteq\mathsf{p}$. If $v(q)=1$, then clearly $\{v\}\models \mathsf{x}\top\subseteq\mathsf{p}q$ so $T\models \mathsf{x}\top\subseteq\mathsf{p}q$ and we can conlude that $T\models \chi$. Otherwise,  $v(q)=0$ and $\{v\}\models \mathsf{x}\bot\subseteq\mathsf{p}q$ and we can again conlude that $T\models \chi$.
 \end{proof}

As is commonly done for propositional team-based logics, we show that each formula is provably equivalent to some formula in the logic's normal form.

\begin{lemma} \label[lemma]{quc derivable nf}
    For $\phi\in\mathcal{L}_{qu}$, there is a team property $\mathcal{C}$, that either is $\{\emptyset\}$ or does not contain the empty set, such that $\phi\dashv\vdash \Psi^{\prime}_\mathcal{C}=\Glor_{T\in\mathcal{C}}\psi^{\prime}_T$.
\end{lemma}
\begin{proof}
    \begin{enumerate}[label=\textbf{-}]
        \item If $\phi=\bot$, then let $\mathcal{C}=\{\emptyset\}$. Now $\Psi^{\prime}_\mathcal{C}=\bot$ and the derivation is trivial.

\item  If $\phi=\mathsf{y}\subseteq\mathsf{q}$, let $\mathcal{C}=\lVert\mathsf{y}\subseteq \mathsf{q}\rVert\setminus\{\emptyset\}$. We show first that $\Psi^{\prime}_{\mathcal{C}}\vdash \mathsf{y}\subseteq \mathsf{q}$. By $\glor$E, it suffices to show that $\psi^{\prime}_T\vdash \mathsf{y}\subseteq \mathsf{q}$ for each $T\in\mathcal{C}$.  Let $T\in \mathcal{C}$. Since $T\neq\emptyset$, there is a valuation $v\in T$ such that $v(\mathsf{q})=v(\mathsf{y})$, hence $\psi^{\prime}_T\vdash \mathsf{y}\subseteq \mathsf{q}$ by $\land$E, $\subseteq$Perm and $\subseteq$Proj. 

For the other direction, we show that 
$\mathsf{y}\subseteq \mathsf{q}\vdash \Psi^{\prime}_{\mathcal{C}}$. By $\subseteq$Ext, $\glor$I and $\subseteq$Perm, we have $\mathsf{y}\subseteq \mathsf{q}\vdash \Glor_{\mathsf{x}\in\{\top,\bot\}^{|r|}} \mathsf{y}\mathsf{x}\subseteq \mathsf{q}\mathsf{r}$, where $\mathsf{q}\mathsf{r}=\mathsf{p}$. We note that each disjunct and  $\mathsf{y}\mathsf{x}$ correspond to a unique valuation for which $\{v\}\in\mathcal{C}$, so $\Glor_{\mathsf{x}\in\{\top,\bot\}^{|r|}} \mathsf{y}\mathsf{x}\subseteq \mathsf{q}\mathsf{r}=\Glor_{\{v\}\in\mathcal{C}}\psi^{\prime}_{\{v\}}\vdash\Glor_{T\in\mathcal{C}}\psi^{\prime}_{T}=\Psi^{\prime}_{\mathcal{C}}$ follows by $\glor$I. 

\item[(IH)] Let $\mathcal{D}_1,\mathcal{D}_2$ be nonempty team properties such that $\phi_1\dashv\vdash \Psi^{\prime}_{\mathcal{D}_1}$ and $\phi_2\dashv\vdash \Psi^{\prime}_{\mathcal{D}_2}$. Assume additionally that each of $\mathcal{D}_1,\mathcal{D}_2$ is either $\{\emptyset\}$ or does not contain the empty set.
        
        \item  Let $\phi=\phi_1\glor\phi_2$. If $D_2=\{\emptyset\}$, then we derive $\Psi^{\prime}_{\mathcal{D}_1}\glor\Psi^{\prime}_{\{\emptyset\}}=\Psi^{\prime}_{\mathcal{D}_1}\glor\bot\vdashv\Psi^{\prime}_{\mathcal{D}_1}$ by $\glor$I, $\glor$E and $\bot$E. The derivation when ${\mathcal{D}_1}=\{\emptyset\}$ is similar. So suppose that neither team property contains the empty team. By the induction hypothesis, $\glor$E and $\glor$I, we have $\phi_1\glor \phi_2\dashv\vdash \Psi^{\prime}_{\mathcal{D}_1}\glor\Psi^{\prime}_{\mathcal{D}_1}= \Psi^{\prime}_{\mathcal{D}_1\cup \mathcal{D}_2}$.
        
        \item  Let $\phi=\phi_1\land\phi_2$. We naturally let $\mathcal{C}=\mathcal{D}_1\cap\mathcal{D}_2$, and observe that by assumption, this intersection either is $\{\emptyset\}$, or does not contain the empty team. If $\mathcal{C}=\{\emptyset\}$, then either $\mathcal{D}_1$ or $\mathcal{D}_2$ is $\{\emptyset\}$. If $\mathcal{D}_2=\{\emptyset\}$, then we derive $\phi_1\land\phi_2\vdashv\Psi^{\prime}_{\mathcal{D}_1}\land \bot\vdashv \bot=\Psi^{\prime}_{\mathcal{D}_2}$ by $\land$E and $\bot$E. The case when $\mathcal{D}_1=\{\emptyset\}$ is similar. Suppose now that neither $\mathcal{D}_1$ nor $\mathcal{D}_2$ contains the empty team. Now $$\mathcal{C}=\{S\mid S\supseteq T_1\cup T_2 \text{ for some } T_1\subseteq\mathcal{D}_1 \text{ and }  T_2\subseteq\mathcal{D}_2\}.$$
        We show first that $\Psi^{\prime}_\mathcal{C}\vdash \phi_1\land\phi_2$. By $\glor$E and the induction hypothesis, it suffices to prove for each $S\in\mathcal{C}$ that $\psi^{\prime}_S\vdash \Psi^{\prime}_{\mathcal{D}_1}\land\Psi^{\prime}_{\mathcal{D}_2}$. By construction of $\mathcal{C}$ there is some $T_1\in \mathcal{D}_1$ for which $S\supseteq T_1$, hence we derive $\psi^{\prime}_S\vdash \psi^{\prime}_{T_1}\vdash \Psi^{\prime}_{\mathcal{D}_1}$ by $\land$E and $\glor$I. The derivation of $\Psi^{\prime}_{\mathcal{D}_2}$ is similar, and we conclude that $\Psi^{\prime}_\mathcal{C}\vdash \Psi^{\prime}_{\mathcal{D}_1}\land\Psi^{\prime}_{\mathcal{D}_2}$.

        For the other direction, by the induction hypothesis, it suffices to derive $\Psi^{\prime}_{\mathcal{D}_1}, \Psi^{\prime}_{\mathcal{D}_2}\vdash \Psi^{\prime}_{\mathcal{C}}$. By $\glor$E, this reduces to showing that $\psi^{\prime}_{T_1},\psi^{\prime}_{T_2}\vdash \Psi^{\prime}_{\mathcal{C}}$ for all $T_1\in \mathcal{D}_1$ and $T_2\in \mathcal{D}_2$. 
        Since $T_1\cup T_2\in\mathcal{C}$ and $S\supseteq T_1\cup T_2$, we obtain the desired derivation $\psi^{\prime}_{T_1},\psi^{\prime}_{T_2}\vdash\psi^{\prime}_{T_1\cup T_2}\vdash  \Psi^{\prime}_{\mathcal{C}}$ using $\land$I and $\glor$I. 
    \end{enumerate}
\end{proof}

With the result of \Cref{quc derivable nf}, proving completeness of the system is straightforward. 

\begin{theorem}  The rules in \Cref{rules:wuc_logic} form a complete proof system for $\mathcal{L}_{qu}$, i.e., for a set of $\mathcal{L}_{qu}$-formulas $\Gamma\cup\{\phi\}$, if $\Gamma\models\phi$ then $\Gamma\vdash \phi$.
\end{theorem}
\begin{proof}
By compactness, it suffices to show that if $\gamma\models\phi$, then $\gamma\vdash \phi$.  By \Cref{quc derivable nf}, there are nonempty team properties $\mathcal{D}$ and $\mathcal{C}$ such that $\gamma\dashv\vdash\Psi^{\prime}_{\mathcal{D}}$ and $\phi\dashv\vdash\Psi^{\prime}_{\mathcal{C}}$. By soundness, $\Psi^{\prime}_{\mathcal{D}}\models \Psi^{\prime}_{\mathcal{C}}$. If $\Psi^{\prime}_{\mathcal{D}}=\bot$ we conclude the proof by $\bot$E, and if $\Psi^{\prime}_{\mathcal{C}}=\bot$, also $\Psi^{\prime}_{\mathcal{D}}=\bot$ and the result is immediate. Suppose that neither normal form is $\bot$, then by construction of the normal forms, we can assume that neither $\mathcal{C}$ nor $\mathcal{D}$ contains the empty team. By \Cref{wuc lm 1} \cref{wuc lm 2 item}, there is some nonempty team $T\in \mathcal{C}$ such that $S\supseteq T$, thus we derive $\psi^{\prime}_{S}\vdash_{\land E} \psi^{\prime}_{T}\vdash_{\glor I}\Psi^{\prime}_{\mathcal{C}}$. Now
$\Psi^{\prime}_{\mathcal{D}}\vdash\Psi^{\prime}_{\mathcal{C}}$ follows by $\glor$E. We conclude $\gamma\vdash \phi$.
\end{proof}

We end this subsection on $\mathcal{L}_{qu}$ with an example derivation.

\begin{example}
    We can derive $\vdash \top\subseteq q\glor \bot\subseteq q$ in the system for $\mathcal{L}_{qu}$.
    \bigbreak
    
    \centering\AxiomC{$\sempty\subseteq\sempty$}

\AxiomC{$[\top\subseteq q]$}
\RightLabel{$\glor$I}
\UnaryInfC{$\top\subseteq q\glor \bot\subseteq q$}

\AxiomC{$[\bot\subseteq q]$}
\RightLabel{$\glor$I}
\UnaryInfC{$\top\subseteq q\glor \bot\subseteq q$}

\RightLabel{$\subseteq$Ext}
\TrinaryInfC{$\top\subseteq q\glor \bot\subseteq q$}
\DisplayProof
\end{example}

\subsection{Upward closed logic}

We show that the logic $\mathcal{L}_{u}$ is expressively complete for all upward closed team properties and provide a complete axiomatization. 

Recall the syntax of the upward closed logic $\mathcal{L}_{u}$, \[\phi::= \top \mid \mathsf{x}\subseteqq \mathsf{p} \mid (\phi\land\phi)\mid (\phi\glor\phi).\]

%


No formula in the $\top$-free fragment of the logic is satisfied by the empty team, hence $\mathcal{L}_{u}$ does not have the empty team property. Moreover, $\sempty\subseteqq \sempty$ is not equivalent with $\top$, since $T\models\sempty\subseteqq \sempty$ if and only if $T\neq\emptyset$. Recall that we restrict the variables in the sequence $\mathsf{p}$ of $\mathsf{x}\subseteqq \mathsf{p}$, in this way, all formulas in the logic are satisfied by at least one team. 

We have already seen that the logic can only express upward closed team properties; we aim to show that it can express all of them by using the logic's normal form.  

Let the sequence $\mathsf{p}$ contain all propositional symbols in $\mathbb{P}$. For any team $T$, define $$\psi^{*}_T:=\bigwedge_{v\in T} \mathsf{x}^v\subseteqq \mathsf{p}.$$ 

Note that now $\psi^{*}_\emptyset=\bigwedge\emptyset=\top$, which is as wanted since for an upward closed team property $\mathcal{C}$, we have that if $\emptyset\in\mathcal{C}$, then $\mathcal{C}$ contains all teams over $\mathbb{P}$. 

The formulas $\psi^{*}_T$ behave similarly to the corresponding ones in the quasi upward closed setting, but now without special cases for the empty team. 

\begin{lemma} \label[lemma]{uc lm 1 B}
Let $\mathcal{C}$ and $\mathcal{D}$ be nonempty team properties.
\begin{enumerate}[label =$($\roman*$)$]


      \item
$S\models \psi^{*}_T$ iff  $S\supseteq T$. \label{uc lm 1 item B}

\label{uc lm 2 item A}
              \item $\Glor_{S\in\mathcal{D}}\psi^{*}_S\models \Glor_{T\in\mathcal{C}}\psi^{*}_T$ iff for each $S\in\mathcal{D}$ there is some $T\in\mathcal{C}$ such that $S\supseteq T$. \label{uc lm 2 item B}
  

\end{enumerate} 
\end{lemma}

\begin{proof}
  The proofs are straightforward and similar to the ones in \Cref{wuc lm 1}.
\end{proof}

The expressive completeness result is now obtained by considering the normal forms $$\Psi^{*}_\mathcal{C}:=\Glor_{T\in\mathcal{C}}\psi^{*}_T,$$ where $\mathcal{C}$ is a nonempty team property.

\begin{theorem}
    $\mathcal{L}_{u}$ is expressively complete for all nonempty upward closed team properties $\mathcal{C}$.
\end{theorem}
\begin{proof}
    By upward closure of the logic, for any $\phi\in \mathcal{L}_{u}$, there is some nonempty $\mathcal{C}$ such that  $\mathcal{C}=\lVert\phi\rVert$.

    For the other direction, let $\mathcal{C}$ be a nonempty upward closed team property and consider the formula $\Psi^{*}_\mathcal{C}=\Glor_{T\in\mathcal{C}}\psi^{*}_T$.

    For any team $S$ over $\mathbb{P}$, we have 
\begin{align*}
  S \models  \Glor_{T \in \mathcal{C}} \psi^{*}_T \iff & S\models \psi^{*}_T \text{ for some } T\in \mathcal{C}  \\
    \iff & S\supseteq T \text{ for some } T\in \mathcal{C} \\
  \iff & S \in \mathcal{C},
\end{align*}
where the second equivalence is by  \Cref{uc lm 1 B} \cref{uc lm 1 item B}, and the last equivalence is by We conclude $\mathcal{C}=\lVert\Glor_{T \in \mathcal{C}} \psi^{*}_{T}\rVert$.
\end{proof}

We present a proof system for the logic $\mathcal{L}_{u}$ in \Cref{rules:uc_logic}. The system is very similar to that of $\mathcal{L}_{qu}$, obtained by omitting the rule $\bot$E, replacing $\sempty\subseteq\sempty$ with $\top$ in the rule $\top$I, and finally replacing the inclusion symbol $\subseteq$ with the nonempty variant $\subseteqq$. The rules for the connectives $\land$ and $\glor$ remain unchanged.

\begin{table}[bht]\centering \small
  \renewcommand*{\arraystretch}{3.8}
\begin{tabular*}{\linewidth}{@{\extracolsep{\fill}}|ccc|}
\hline

\multicolumn{3}{|c|}{
     \AxiomC{} 
  \RightLabel{$\top$I}
  \UnaryInfC{$\top$}
\DisplayProof
\hspace{2cm}
The rules $\land$I, $\land$E, $\glor$I and $\glor$E from \Cref{rules:wuc_logic}.}
 \\ [1ex] 

 \AxiomC{}
     \noLine  
     \UnaryInfC{$\mathsf{x}y\subseteqq\mathsf{p}q$}
     \RightLabel{$\subseteqq$Proj}
\UnaryInfC{$\mathsf{x}\subseteqq\mathsf{p}$}
\DisplayProof
&

 \AxiomC{}
     \noLine  
     \UnaryInfC{$\mathsf{x}\mathsf{y}\mathsf{z}\subseteqq\mathsf{u}\mathsf{v}\mathsf{w}$}
     \RightLabel{$\subseteqq$Perm (1)}
\UnaryInfC{$\mathsf{x}\mathsf{z}\mathsf{y}\subseteqq\mathsf{u}\mathsf{w}\mathsf{v}$} 
 \DisplayProof
& 
\AxiomC{}
 \noLine
 \UnaryInfC{$\mathsf{x}\subseteqq\mathsf{p}$}

 \AxiomC{$[\mathsf{x}\top\subseteqq\mathsf{p}q]$}
 \noLine
 \UnaryInfC{\vdots}
 \noLine
 \UnaryInfC{$\chi$} 
 
\AxiomC{$[\mathsf{x}\bot\subseteqq\mathsf{p}q]$}
 \noLine
 \UnaryInfC{\vdots}
 \noLine
 \UnaryInfC{$\chi$} 
 \RightLabel{$\subseteqq$Ext}
 \TrinaryInfC{$\chi$}
 \DisplayProof
 \\ [4ex] 
 
 \hline

\multicolumn{3}{|c|}{(1) Provided that $|y|=|v|$ and $|z|=|w|$.} \\
 \hline
\end{tabular*}
 \caption{System for $\mathcal{L}_u$.}
\label{rules:uc_logic}
\end{table}

\begin{theorem}
       The rules in \Cref{rules:uc_logic}
 are sound for upward closed formulas.
\end{theorem}

\begin{proof}
    Routine proof, where in particular, showing soundness of $\subseteqq$Ext is similar to the soundness proof for $\subseteq$Ext in \Cref{wuc_soundness}.
\end{proof}

We show that each formula is provably equivalent to some formula in the logic's normal form.

\begin{lemma} \label[lemma]{uc derivable nf B}
    For $\phi\in \mathcal{L}_{u}$, there is a nonempty team property $\mathcal{C}$ such that $\phi\dashv\vdash \Psi^{*}_\mathcal{C}=\Glor_{T\in\mathcal{C}}\psi^{*}_T$.
\end{lemma}
\begin{proof}


\begin{enumerate}[label=\textbf{-}]
        \item  If $\phi=\top$, then let $\mathcal{C}=\{\emptyset\}$. Now $\Psi^{*}_\mathcal{C}=\psi^{*}_\emptyset=\top$, hence the derivation is trivial.

\item If $\phi=\mathsf{y}\subseteqq\mathsf{q}$, let $\mathcal{D}=\lVert\mathsf{y}\subseteqq \mathsf{q}\rVert$. We show first that $\Psi^{*}_{\mathcal{D}}\vdash \mathsf{y}\subseteqq \mathsf{q}$.
Note that $\emptyset\not\in \mathcal{D}$. Hence each $T\in\mathcal{D}$ has a valuation $v\in T$ such that $v(\mathsf{q})=v(\mathsf{y})$, thus we derive $\mathsf{y}\subseteqq \mathsf{q}$  by $\glor$E, $\land$E, $\subseteqq$Perm and  $\subseteqq$Proj. 

For the other direction, use the arguments of the corresponding $\subseteq$-case in the proof of \Cref{quc derivable nf}, using the rule $\subseteqq$Ext instead of $\subseteq$Ext.
%

\item[(IH)] Let $\mathcal{D}_1,\mathcal{D}_2$ be nonempty team properties such that $\phi_1\dashv\vdash \Psi^{*}_{\mathcal{D}_1}$ and $\phi_2\dashv\vdash \Psi^{*}_{\mathcal{D}_2}$.
        
        \item  If $\phi=\phi_1\glor\phi_2$, then  by $\glor$E, $\glor$I and the induction hypothesis, we have $\phi_1\glor \phi_2\dashv\vdash \Psi^{*}_{\mathcal{D}_1}\glor\Psi^{*}_{\mathcal{D}_1}= \Psi^{*}_{\mathcal{D}_1\cup \mathcal{D}_2}$.
        
        \item  If $\phi=\phi_1\land\phi_2$, let $$\mathcal{C}=\{T\mid S_1\cup S_2\subseteq T \text{ for some } S_1\subseteq\mathcal{D}_1 \text{ and } S_2\subseteq\mathcal{D}_2\}.$$
We then reason like in the corresponding $\land$-case in the proof of \Cref{quc derivable nf}, without the special cases relating to the empty team. 
%
    \end{enumerate}
\end{proof}

We can now prove completeness of the system for the upward closed logic $\mathcal{L}_u$.

\begin{theorem}
  The rules in \Cref{rules:uc_logic} form a complete proof system for $\mathcal{L}_{u}$, i.e., for a set of $\mathcal{L}_{u}$-formulas $\Gamma\cup\{\phi\}$, if $\Gamma\models\phi$ then $\Gamma\vdash \phi$.
\end{theorem}
\begin{proof}
By compactness, it suffices to show that if $\gamma\models\phi$, then $\gamma\vdash \phi$.  By \Cref{uc derivable nf B}, there are nonempty team properties $\mathcal{D}_1$ and $\mathcal{D}_2$ such that $\gamma\dashv\vdash\Psi^{*}_{\mathcal{D}_1}$ and $\phi\dashv\vdash\Psi^{*}_{\mathcal{D}_2}$. By soundness, $\Psi^{*}_{\mathcal{D}_1}\models \Psi^{*}_{\mathcal{D}_2}$. 
%
%
By \Cref{uc lm 1 B} \cref{uc lm 2 item B}, for each $S\in \mathcal{D}$, there is some $T\in \mathcal{C}$ such that $S\supseteq T$. Thus we derive $\psi^{*}_{S}\vdash \psi^{*}_{T}\vdash_{\glor I}\Psi^{*}_{\mathcal{C}}$, where the first step is trivial if $S\in \mathcal{D}$, and otherwise derived by either $\land$E, or 
$\top$I if $\emptyset\in \mathcal{C}$. 
Hence $\Psi^{*}_{\mathcal{D}}\vdash\Psi^{*}_{\mathcal{C}}$ by $\glor$E. We conclude $\gamma\vdash \phi$. 
\end{proof}

We end this subsection on $\mathcal{L}_{u}$ with an example derivation.

\begin{example}
    We can derive $\top\subseteq p\vdash \top\top\subseteqq pq\glor \top\bot\subseteqq pq$ in the system for $\mathcal{L}_{u}$. The derivation is similar to that of $\vdash \top\subseteq q\glor \bot\subseteq q$ in the system for $\mathcal{L}_{qu}$, but note here that $\emptyset\not\models \top\subseteqq q\glor \bot\subseteqq q$.
    
\medskip

\centering
\AxiomC{$\mathsf{x}\subseteqq\mathsf{p}$}

\AxiomC{$[\mathsf{x}\top\subseteqq\mathsf{p}q]$}
\RightLabel{$\glor$I}
\UnaryInfC{$\mathsf{x}\top\subseteqq\mathsf{p}q\glor \mathsf{x}\bot\subseteqq\mathsf{p}q$}

\AxiomC{$[\mathsf{x}\bot\subseteqq\mathsf{p}q]$}
\RightLabel{$\glor$I}
\UnaryInfC{$\mathsf{x}\top\subseteqq\mathsf{p}q\glor \mathsf{x}\bot\subseteqq\mathsf{p}q$}

\RightLabel{$\subseteq$Ext}
\TrinaryInfC{$\mathsf{x}\top\subseteqq\mathsf{p}q\glor \mathsf{x}\bot\subseteqq\mathsf{p}q$}
\DisplayProof

\end{example}

\subsection{Quasi downward closed logic}

We move from the (quasi) upward closed settings to the quasi downward closed setting, and aim to show that the logic $\mathcal{L}_{qd}$ is expressively complete for all quasi downward closed team properties. We also introduce a sound and complete proof system for the logic by modifying the one for $\PLglor$ in \cite{Yang16}. 
Recall the syntax of 
the logic $\mathcal{L}_{qd}$, 
\[\phi::= \bullet \mid \mathsf{p}\fullsubseteq \mathsf{x} \mid (\phi\land\phi) \mid (\phi\lor\phi)\mid (\phi\glor\phi).\]

We have already seen that this logic has the full team property and is quasi downward closed. Note that due to the full team atom $\bullet$, the logic does not have the empty team property. However, the $\bullet$-free fragment has the empty team property. With the full dual primitive inclusion atom, we have that $\lVert\mathsf{p}\fullsubseteq\mathsf{x}\rVert=\lVert\mathsf{p}^\mathsf{x}\rVert\cup\{\mathbb{F}\}$, and $\lVert p_1p_1\fullsubseteq\top\bot\rVert=\{\emptyset,\mathbb{F}\}$.

Next, we adapt the normal form for the logic $PD^{\glor}$ from \cite{Yang16} to our logic, while taking special care of the full team. First, we define formulas that capture subteams (modulo the full team).
 For a team $T$ and a sequence $\mathsf{p}$ containing all variables from $\mathbb{P}$, define 
$$\theta^{\prime}_{T}:= \bigvee_{v\in T} \mathsf{p}\fullsubseteq \mathsf{x}^v,$$
where we stipulate $\theta^{\prime}_{\emptyset}:= p_1\mathsf{p}\fullsubseteq\top\bot\dots\bot$. Thus, for all teams $T$, the formula $\theta^{\prime}_T$ is union closed and has the empty team property.

\begin{lemma} (Essentially \cite{Yang16}) \label[lemma]{lemma wdc}
  Let $T$ and $S$ be teams over the variables in $\mathbb{P}$ and let $\mathcal{D}$ and $\mathcal{C}$ be nonempty team properties not containing the full team $\mathbb{F}$.
  
  \begin{enumerate}[label=(\roman*)]
      \item $S\models \theta^{\prime}_T$ iff $S\subseteq T$ or $S=\mathbb{F}$. \label{lemma wdc i}
      \item $\Glor_{S\in\mathcal{D}}\theta^{\prime}_S\models \Glor_{T\in\mathcal{C}}\theta^{\prime}_T$ iff for each $S\in\mathcal{D}$ there exists $T\in\mathcal{C}$ such that $S\subseteq T$.\label{lemma wdc ii}
  \end{enumerate}
\end{lemma}
\begin{proof}
  \begin{enumerate}[label=(\roman*)]
  \item Clearly, if $S=\mathbb{F}$ or $S=\emptyset$, then both sides of the equivalence holds. Suppose that neither $S\neq\mathbb{F}$ nor $S\neq\emptyset$. Then  $S\models \bigvee_{v\in T} \mathsf{p}\fullsubseteq \mathsf{x}^v$ if and only if each $v'\in S$ is such that $\{v'\}\models \mathsf{p}\fullsubseteq \mathsf{x}^v$ for some $v\in T$,  which is only the case when $v'=v$. Hence, the equivalence with $S\subseteq T$ is obtained.

\item Suppose that $\mathcal{D}$ and $\mathcal{C}$ are nonempty team properties that do not contain the full team $\mathbb{F}$. Suppose first that $\Glor_{S\in\mathcal{D}}\theta^{\prime}_S\models \Glor_{T\in\mathcal{C}}\theta^{\prime}_T$. Observe that for each $S\in\mathcal{D}$, $S\models\theta^{\prime}_S$ follows by \cref{lemma wdc i}, hence $S\models \Glor_{S\in\mathcal{D}}\theta^{\prime}_S$. Thus by assumption, $S\models \Glor_{T\in\mathcal{C}}\theta^{\prime}_T$ follows, hence $S\models \theta^{\prime}_{T}$ for some $T\in \mathcal{C}$, for which $S\subseteq T$ by \cref{lemma wdc i}.

For the other direction, let $S'\models \Glor_{S\in\mathcal{D}}\theta^{\prime}_S$. We show that $S'\models \Glor_{T\in\mathcal{C}}\theta^{\prime}_T$. If $S'=\mathbb{F}$, we are done by the full team property, so suppose that $S'\neq \mathbb{F}$. Then from \cref{lemma wdc i} and the fact that $S'\models \theta^{\prime}_S$ for some $S\in\mathcal{D}$, it follows that $S'\subseteq S$. By assumption, there is some $T\in\mathcal{C}$ such that $S'\subseteq S\subseteq T$. Again, by \cref{lemma wdc i} we have that $S'\models \theta^{\prime}_T$. We conclude that $S'\models\Glor_{T\in\mathcal{C}}\theta^{\prime}_T$. 
  \end{enumerate}
\end{proof}

We now define the normal form for $\mathcal{C}$ that does not contain the full team, by 
$$\Theta^{\prime}_\mathcal{C}:=\Glor_{T\in\mathcal{C}}\theta^{\prime}_T,$$
and stipulate $\Theta^{\prime}_{\{\mathbb{F}\}}:=\bullet$.

 We adapt the expressive completeness proof for the downward closed logic $PD^{\glor}$ from \cite{Yang16}. 

\begin{theorem}
    $\mathcal{L}_{qd}$ is expressively complete for all 
    quasi downward closed team properties.
\end{theorem}
\begin{proof}
 Every $\phi\in\mathcal{L}_{qd}$ defines a quasi downward closed property that contains the full team, so it remains to show that any quasi downward closed property $\mathcal{C}$ over $\mathbb{P}$ is definable by some formula $\phi\in \mathcal{L}_{qd}$, i.e.,   
$\mathcal{C} = \lVert\phi\rVert$.  

If $\mathcal{C}=\{\mathbb{F}\}$, then the formula $\Theta^{\prime}_\mathcal{C}=\bullet$ clearly defines the property. So suppose that $\{\mathbb{F}\}\subsetneq\mathcal{C}$. Now 
for any team $S$ over $\mathbb{P}$, we have 
\begin{align*}
  S \models  \Glor_{T \in \mathcal{C}\setminus\{\mathbb{F}\}} \theta^{\prime}_{T} \iff & S\models \theta^{\prime}_T \text{ for some } T\in \mathcal{C}\setminus\{\mathbb{F}\}  \\
    \iff & S=\mathsf{F } \text{ or } S\subseteq T \text{ for some } T\in \mathcal{C}\setminus\{\mathbb{F}\} \\
  \iff & S \in \mathcal{C},
\end{align*}
where the second equivalence is by  \Cref{lemma wdc} \cref{lemma wdc i}. Thus $\mathcal{C}=\lVert\Glor_{T \in \mathcal{C}\setminus\{\mathbb{F}\}} \theta^{\prime}_{T}\rVert$, as desired.
\end{proof}

Next, we define a proof system for $\mathcal{L}_{qd}$ in \Cref{rules:wdc_logic}, inspired by the one for $\PLglor$ in \cite{Yang16}.

\begin{table}[ht]\centering \small
  \renewcommand*{\arraystretch}{3.8}
\begin{tabular*}{\linewidth}{@{\extracolsep{\fill}}|ccc|}
\hline

    \AxiomC{} 
     \noLine     
     \UnaryInfC{$\phi\lor\bullet$}
     \RightLabel{$\bullet$I}
     \UnaryInfC{$\bullet$}
\DisplayProof

&
 
     \AxiomC{} 
     \noLine     
     \UnaryInfC{$\bullet$}
     \RightLabel{$\bullet$E}
     \UnaryInfC{$\phi$}
\DisplayProof
&

\hspace{.5cm}

 \\ [2ex]

 \AxiomC{}
\RightLabel{$\top$I}
\UnaryInfC{$\emptyset\fullsubseteq\emptyset$}
\DisplayProof
& 
\AxiomC{}
\noLine
\UnaryInfC{$q\fullsubseteq\top$}
\AxiomC{}
\noLine
\UnaryInfC{$q\fullsubseteq\bot$}
\RightLabel{$\bot$E (1)}
\BinaryInfC{$\psi$}
\DisplayProof
&
\AxiomC{}
\noLine
\UnaryInfC{$\psi\lor qq\fullsubseteq\top\bot$}
\RightLabel{$\bot\lor$E}
\UnaryInfC{$\psi$}
\DisplayProof
\\[2ex] 

 \AxiomC{}
     \noLine  
     \UnaryInfC{$\mathsf{p}q\fullsubseteq\mathsf{x}y$}
     \RightLabel{$\fullsubseteq$Proj}
\UnaryInfC{$\mathsf{p}\fullsubseteq\mathsf{x}$}
\DisplayProof
&

 \AxiomC{}
     \noLine  
     \UnaryInfC{$\mathsf{u}\mathsf{v}\mathsf{w} \fullsubseteq \mathsf{x}\mathsf{y}\mathsf{z}$}
     \RightLabel{$\fullsubseteq$Perm (2)}
\UnaryInfC{$\mathsf{u}\mathsf{w}\mathsf{v}\fullsubseteq\mathsf{x}\mathsf{z}\mathsf{y}$} 
 \DisplayProof
& 
 \AxiomC{}
     \noLine  
     \UnaryInfC{$\mathsf{p}\fullsubseteq\mathsf{x}$}
     \RightLabel{$\fullsubseteq$Ext}
\UnaryInfC{$\mathsf{p}q\fullsubseteq\mathsf{x}\top\lor\mathsf{p}q\fullsubseteq\mathsf{x}\bot$}
\DisplayProof
 \\ [3ex] 

  \hline

\multicolumn{3}{|c|}{
The rules $\land$I, $\land$E, $\glor$I and $\glor$E from \Cref{rules:wuc_logic}.}
 \\ [2ex]

\AxiomC{}
\noLine
\UnaryInfC{$\phi$}
\RightLabel{$\lor$I (1)}
\UnaryInfC{$\phi\lor\psi$}
\DisplayProof
& 
\AxiomC{}
\noLine
\UnaryInfC{$\phi$}
\RightLabel{$\lor$I (1)}
\UnaryInfC{$\psi\lor\phi$}
\DisplayProof
& 
\AxiomC{}
\noLine
\UnaryInfC{$\phi \lor \psi$}
\noLine
\AxiomC{[$\phi$]}
\noLine
\UnaryInfC{\vdots}
\noLine
\UnaryInfC{$\chi$}
\AxiomC{[$\psi$]}
\noLine
\UnaryInfC{\vdots}
\noLine
\UnaryInfC{$\chi$}
\RightLabel{$\lor$E (3)}
\TrinaryInfC{$\chi$}
\DisplayProof\\ [2ex] 
 \AxiomC{}
     \noLine  
     \UnaryInfC{$\phi\lor\psi$}
     \RightLabel{$\lor$Com}
\UnaryInfC{$\psi\lor\phi$}
\DisplayProof
&
\AxiomC{}
 \noLine
 \UnaryInfC{$\phi\lor\psi$}
 \AxiomC{$[\psi]$}
 \noLine
 \UnaryInfC{\vdots}
 \noLine
 \UnaryInfC{$\gamma$} \RightLabel{$\lor$Mon}
 \BinaryInfC{$\phi\lor\gamma$}
 \DisplayProof

  &
  \AxiomC{}
  \noLine
  \UnaryInfC{$\phi\lor(\psi\glor\theta)$}
  \RightLabel{$\lor\glor$Distr}
 \UnaryInfC{$(\phi\lor\psi)\glor(\phi\lor\theta)$}  
 \DisplayProof
\\[2ex] 
 
\hline

 \multicolumn{3}{|l|}{\makecell{(1) $\psi$ is $\bullet$-free. (2) Provided that $|y|=|v|$ and $|z|=|w|$.  (3) $\chi$ is $\glor$-free.}}\\
 
\hline

\end{tabular*}
 \caption{System for $\mathcal{L}_{qd}$.}
\label{rules:wdc_logic}
\end{table}

The rules for the connectives are identical to the system for $\PLglor$ in \cite{Yang16}, except in two aspects. One, the rule deriving $\phi\lor(\psi\lor\chi)\vdash (\phi\lor\psi)\lor\chi$ which can be omitted also from the other system (as observed in \cite{Yang26}); in the process of proving completeness of the system $\PLglor$ in \cite{Yang16}, it is only applied to formulas in $PL$, for which it is derivable by the rules $\lor$E and $\lor$I. Two, for soundness, we restrict the disjunct introduced by $\lor$I to the fragment of $\mathcal{L}_{qd}$ that has the empty team property, i.e., $\bullet$-free. Similarly to their system, soundness of $\lor$E relies on the conclusion being union closed, which the $\glor$-free fragment of $\mathcal{L}_{qd}$ is.

Our rule $\bot\lor$E is a translation of the corresponding rule in \cite{Yang16} to the setting with the full team property, seen through the equality  $\lVert qq\fullsubseteq \top\bot\rVert=\lVert \bot\rVert\cup\{\mathbb{F}\}$. Similarly, the rule $\bot$E translates the corresponding one in \cite{Yang16} to our setting, noticing that 
 $\lVert q\fullsubseteq \top\land q\fullsubseteq\bot\rVert=\{\emptyset,\mathbb{F}\}$, thus for soundness to hold, we also here assume that the conclusion of the rule is from the $\bullet$-free fragment of the logic.  
By $\top$I and $\subseteq$Ext we can derive the law of excluded middle, which with our notation is $q\fullsubseteq\top\lor q\fullsubseteq\bot$. 

We add two novel rules to the system to handle the full dual primitive inclusion atoms: $\fullsubseteq$Proj and $\fullsubseteq$Perm. This is simply to be able to eliminate the hidden conjuncts in the full dual primitive inclusion atom, for instance: $qr\fullsubseteq \bot\top\vdash rq\fullsubseteq\top\bot\vdash r\fullsubseteq\top$, which in $\PLglor$ corresponds to the following derivation $\neg q\land r\vdash \neg q$, derivable by $\land$E. 

Lastly, we add two novel rules to handle $\bullet$: $\bullet$I allowing us to eliminate $\bullet$ from a split-disjunction, and $\bullet$E allowing us to derive any formula from $\bullet$ thanks to the full team property of the logic. 

\begin{theorem} \label{wdc_soundness}
    The rules in \Cref{rules:wdc_logic}
 are sound for quasi downward closed formulas. \end{theorem}
 \begin{proof}
     Routine proof, we include the cases for $\bullet$I and $\fullsubseteq$Ext. 
    First, we show that $\psi\lor\bullet\vdash\bullet$
     If $T\models \psi\lor\bullet$, then there are subteams $T_1,T_2$ of $T$ such that $T_1\cup T_2=T$ and $T_1\models\phi$ and $T_2\models \bullet$. Thus $T_2=\mathbb{F}$ implying that $T=\mathbb{F}$, from which $T\models\bullet$ follows. 

     Next, we show that $\mathsf{p}\fullsubseteq \mathsf{x}\vdash\mathsf{p}q\fullsubseteq\mathsf{x}\top\lor\mathsf{p}q\fullsubseteq\mathsf{x}\bot$. Let $T\models\mathsf{p}\fullsubseteq \mathsf{x}$ and $T\models p_ip_i\fullsubseteq \top\bot$ for some variable $p_i$ from the sequence $p$, then $T$ is either empty or the full team and thus satisfies the conclusion. So suppose that there is no such $p_i$ and that $T$ is neither the empty nor the full team.
     Consider the subteams $T_1=\{v\in T\mid v(q)=1\}$ and $T_2=\{v\in T\mid v(q)=0\}$. Clearly $T_1\cup T_2=T$, $T_1\models \mathsf{p}q\fullsubseteq\mathsf{x}\top$ and  $T_2\models\mathsf{p}q\fullsubseteq\mathsf{x}\bot$. We conclude that $T\models \mathsf{p}q\fullsubseteq\mathsf{x}\top\lor\mathsf{p}q\fullsubseteq\mathsf{x}\bot$. 
 \end{proof}

We state some interesting derivations in the system. 

\begin{lemma}  \label[lemma]{qd lm derivations}
Let $\alpha,\beta_1$ and $\beta_2$ be from the $\bullet$- and $\glor$-free fragment of $\mathcal{L}_{qd}$.

\begin{enumerate}[label=(\roman*)]
    \item $\phi\lor\bullet\vdashv \bullet$. \label{qd lm derivaitons bullet lor}
    \item $\phi\glor\bullet\vdashv \phi$.\label{qd lm derivaitons bullet glor}
    \item $\mathsf{p}\fullsubseteq\mathsf{x},\mathsf{q}\fullsubseteq\mathsf{y}\vdashv \mathsf{pq}\fullsubseteq\mathsf{xy}$.  \hfill ($\fullsubseteq$Aug)
    \item $\mathsf{p}q\fullsubseteq \mathsf{x}y\vdashv\mathsf{p}qq\fullsubseteq \mathsf{x}yy$.
       \item      $(\phi\lor\psi)\glor(\phi\lor\chi)\vdash \phi\lor(\psi\glor\chi)$. \hfill ($\lor\glor\lor Distr$)
    \item $\phi\land(\psi_1\glor\psi_2)\vdashv (\phi\land\psi_1)\glor(\phi\land\psi_2)$ \hfill ($\land\glor$Distr)
      \item $\alpha\land(\beta\lor\beta_2)\vdashv (\alpha\land\beta_1)\lor(\alpha\land\beta_2)$ \hfill ($\land\lor$Distr)
\end{enumerate}

\end{lemma}
\begin{proof}
\begin{enumerate}[label=(\roman*)]
     \item By $\bullet$I and $\bullet$E.
     \item The left-to-right direction is by $\glor$E and $\bullet$E, and the opposite direction by $\glor$I.
    \item The right-to-left direction is by $\fullsubseteq$Perm and $\fullsubseteq$Proj. For the other direction, we use the rules $\fullsubseteq$Ext, $\lor$E, $\fullsubseteq$Perm, $\fullsubseteq$Proj and $\bot$E. We show in detail the derivation for the unary case $p\fullsubseteq\top,q\fullsubseteq\bot\vdash pq\fullsubseteq\top\bot$ in \Cref{derivability of augmentation}. 
\label{lemma derivability of augmentation}
    
\begin{table}[ht] \tiny
  \renewcommand*{\arraystretch}{8}
\hspace*{-3.5cm}\begin{tabular*}{7.96in}{@{\extracolsep{\fill}}|c|}
\hline

\AxiomC{$p\fullsubseteq\top$}
\RightLabel{$\fullsubseteq$Ext}
\UnaryInfC{$pq\fullsubseteq\top\top\lor pq\fullsubseteq\top\bot$}
\AxiomC{$q\fullsubseteq\bot$}
\RightLabel{$\fullsubseteq$Ext,$\fullsubseteq$Perm (1)}
\UnaryInfC{$pq\fullsubseteq\top\bot\lor pq\fullsubseteq\bot\bot$}

\AxiomC{$[pq\fullsubseteq\top\bot]$}

\AxiomC{$[pq\fullsubseteq\top\top]$}
\RightLabel{$\fullsubseteq$Proj,$\fullsubseteq$Perm}
\UnaryInfC{$p\fullsubseteq\top$}
\AxiomC{$[pq\fullsubseteq\bot\bot]$}
\RightLabel{$\fullsubseteq$Proj,$\fullsubseteq$Perm}
\UnaryInfC{$p\fullsubseteq\bot$}
\RightLabel{$\bot$E}
\BinaryInfC{$pq\fullsubseteq\top\bot$}

\RightLabel{$\lor$E}
\TrinaryInfC{$pq\fullsubseteq\top\bot$}

\AxiomC{$[pq\fullsubseteq\top\bot]$}
\RightLabel{$\lor$E}
\TrinaryInfC{$pq\fullsubseteq\top\bot$}
 \vspace*{.2cm}
\DisplayProof \\
 \hline
\end{tabular*}
 \caption{}
\label{derivability of augmentation}
\end{table}  
    \item By $\fullsubseteq$Proj, $\fullsubseteq$Perm and \cref{lemma derivability of augmentation}.
     \item[(v)-(vii)]    As in \cite{Yang16}, using the rules for the connectives.
     %
\end{enumerate}
 
\end{proof}



    


Before the completeness proof, we show that any formula in $\mathcal{L}_{qd}$ is provably equivalent to one in the normal form. 

\begin{lemma} \label[lemma]{qd nf lemma}
        For $\phi\in \mathcal{L}_{qu}$, there is a nonempty team property $\mathcal{C}$ that is either $\mathbb{F}$ or does not contain the full team,  such that $\phi\dashv\vdash \Theta^{\prime}_\mathcal{C}=\Glor_{T\in\mathcal{C}}\theta^{\prime}_T$.
\end{lemma}
\begin{proof}
The case for $\bullet$ and $\fullsubseteq$-atoms are new, while the other induction steps (modulo the handling of subformulas $\bullet$) are essentially due to the corresponding result for $\PLglor$ in \cite{Yang16}.

\begin{enumerate}[label=\textbf{-}]
    \item If $\phi=\bullet$, then for $\Theta^{\prime}_{\{\mathbb{F}\}}=\bullet$, the derivation is trivial. 

    \item Let $\phi=\mathsf{q}\fullsubseteq\mathsf{y}$. If by $\fullsubseteq$Proj and 
    $\bot$E we can obtain $\mathsf{q}\fullsubseteq\mathsf{y}\vdashv rr\fullsubseteq y_iy_j$ with $y_i\neq y_j$, then we also have $rr\fullsubseteq y_iy_j\vdashv p_1\mathsf{p}\fullsubseteq \top\bot\dots\bot$: the left-to right direction is by $\lor$I and $\bot\lor$E, and for the right-to-left direction we use $\fullsubseteq$Proj, $\lor$I and $\bot\lor$E. So suppose not, then $T=\{v\mid v\models \mathsf{q}^\mathsf{x}\}$ is nonempty.
    Now $\mathsf{q}\fullsubseteq\mathsf{y}\vdashv \bigvee_{v\in T}\mathsf{p}\fullsubseteq\mathsf{x}^v=\Theta^{\prime}_{\{t\}}$, where the left-to-right direction follows by $\fullsubseteq$Perm and $\fullsubseteq$Ext, and the right-to-left direction by $\fullsubseteq$Proj, $\fullsubseteq$Perm and $\lor$E. 
    

 \item[(IH)] Let $\mathcal{D}_1,\mathcal{D}_2$ be nonempty team properties such that $\phi_1\dashv\vdash \Theta^{\prime}_{\mathcal{D}_1}$ and $\phi_2\dashv\vdash \Theta^{\prime}_{\mathcal{D}_2}$. Assume additionally that each of $\mathcal{D}_1$ and $\mathcal{D}_2$ either is $\{\mathbb{F}\}$ or does not contain the full set.
 
    \item Let $\phi=\phi_1\glor\phi_2$. If  $\mathcal{D}_2=\{\mathbb{F}\}$, then it suffices to derive $\Theta^{\prime}_{\mathcal{D}_1}\glor\bullet\vdashv \Theta^{\prime}_{\mathcal{D}_1}$, whose left-to-right direction by $\glor$E and $\bullet$E and right-to left by $\glor$I. The case when $\mathcal{D}_1=\{\mathbb{F}\}$ is analogous, so we can assume that neither team property contains the full team. By $\glor$E and $\glor$I, it follows that $\phi_1\glor\phi_2\vdashv\Glor_{T\in\mathcal{C}}\theta^{\prime}_T\glor\Glor_{S\in\mathcal{D}}\theta^{\prime}_S\vdashv \Glor_{T'\in\mathcal{C}\cup\mathcal{D}}\theta^{\prime}_{T'}$. 


        
    \item Let $\phi=\phi_1\lor\phi_2$. If $\mathcal{D}_2=\{\mathbb{F}\}$, we have that  $\phi_1\lor\phi_2\vdashv \phi_1\lor\bullet\vdashv \bullet=\Theta^{\prime}_{\{\mathbb{F}\}}$, with the left-to-right direction by $\lor$Mon and $\bullet$I, and the right-to-left direction by $\bullet$E. The cases when $\mathcal{D}_1=\{\mathbb{F}\}$ is similar. Suppose now that neither $\mathcal{D}_1$ nor $\mathsf{D}_2$ contains the full team. 
  We have the following derivation.
    \begin{align*}
        \phi_1\lor\phi_2
        \vdashv& \Glor_{T\in\mathcal{D}_1}\theta^{\prime}_T\lor \Glor_{S\in\mathcal{D}_2}\theta^{\prime}_S \hspace{2cm} (\lor Mon) \\
        \vdash& \Glor_{T\in\mathcal{D}_1}(\theta^{\prime}_T\lor(\Glor_{S\in\mathcal{D}_2}\theta^{\prime}_S)) \hspace{2cm} (\lor\glor Distr, \lor Com) \\
        \vdash&\Glor_{T\in\mathcal{D}_1}\Glor_{S\in\mathcal{D}_2}(\theta^{\prime}_T\lor\theta^{\prime}_S) \hspace{2cm} (\lor\glor Distr) \\
       \vdash & \Glor_{(T,S)\in\mathcal{D}_1\times\mathcal{D}_2}\theta^{\prime}_{T\cup S}\hspace{2cm} (\lor E,\, \bot\lor E)
    \end{align*}
We use $\lor$E for the last step to remove repeated (split)-disjuncts $\mathsf{p}\fullsubseteq \mathsf{x}^v$ in case $T\cap S\neq\emptyset$, remembering that for all teams $T'$, the characteristic formulas $\theta^{\prime}_{T'}$ are from the $\glor$-free fragment. Similarly, if $T=\emptyset$ or $S=\emptyset$, we use $\bot\lor$E to remove all but one global disjunct $\theta^{\prime}_\emptyset$. 

We can reverse the derivation by first using $\lor$I, then $\lor\glor\lor$Distr, and finally $\lor$Mon. Note that the empty team property required by $\lor$I is satisfied by the formulas.  



\item Let $\phi=\phi_1\land\phi_2$.
If $\mathcal{D}_2=\{\mathbb{F}\}$, we have that  $\phi_1\land\phi_2\vdashv \phi_1\land\bullet\vdashv \bullet=\Theta^{\prime}_{\{\mathbb{F}\}}$ by $\land$E, $\land$I and $\bullet$E. Let us now assume that neither $\mathcal{D}_1$ nor $\mathsf{D}_2$ contains the full team. We show that
$\phi_1\land\phi_2\vdashv\Glor_{T'\in\mathcal{D}_1\cap\mathcal{D}_2}\theta^{\prime}_{T'}$. 

    \begin{align*}
        \phi_1\land\phi_2
        \vdashv& \Glor_{T\in\mathcal{D}_1}\theta^{\prime}_T\land \Glor_{S\in\mathcal{D}_2}\theta^{\prime}_S \hspace{3.7cm} (\land I, \land E) \\        \vdashv&\Glor_{T\in\mathcal{D}_1}\Glor_{S\in\mathcal{D}_2}(\theta^{\prime}_T\land\theta^{\prime}_S) \hspace{3.5cm} (\land\glor Distr) \\
       \vdashv & \Glor_{T\in\mathcal{D}_1}\Glor_{S\in\mathcal{D}_2}\bigvee_{v\in T}\bigvee_{w\in S}(\mathsf{p}\fullsubseteq\mathsf{x}^v\land\mathsf{p}\fullsubseteq\mathsf{x}^w) \hspace{1cm} (\lor I, \lor E, \land\lor Distr) \\
       \vdash & \Glor_{T\in\mathcal{D}_1}\Glor_{S\in\mathcal{D}_2}\bigvee_{v'\in T\cap S}(\mathsf{p}\fullsubseteq\mathsf{x}^{v'}\land\mathsf{p}\fullsubseteq\mathsf{x}^{v'}) \hspace{1cm} (\bot\lor E (\fullsubseteq Proj/Perm/Aug)) \\
         \vdash & \Glor_{T\in\mathcal{D}_1}\Glor_{S\in\mathcal{D}_2}\bigvee_{v'\in T\cap S}(\mathsf{p}\fullsubseteq\mathsf{x}^{v'}) \hspace{2.5cm} (\land E) \\
        = & \Glor_{(T,S)\in\mathcal{D}_1\times\mathcal{D}_2}\theta^{\prime}_{T\cap S}.
    \end{align*}
    Where in the second to last step, we eliminate all split-disjuncts of the form $(\mathsf{p}\fullsubseteq\mathsf{x}^v\land\mathsf{p}\fullsubseteq\mathsf{x}^w)$ for which $v\neq w$. For the right-to-left direction, reverse the first steps by using $\land I$ and $\lor$I.
\end{enumerate}
\end{proof}

We are now ready to prove completeness of the system.

\begin{theorem}
Let $\Sigma\cup\{\psi\}$ be a set if formulas in $\mathcal{L}_{qd}$. Then $\Sigma\models\psi$ implies $\Sigma\vdash\psi$.
\end{theorem}
\begin{proof}
By compactness, it suffices to prove that if $\gamma\models\phi$, then $\gamma\vdash\phi$. So suppose that $\gamma\models\phi$. Then by \Cref{qd nf lemma} there are nonempty team properties $\mathcal{D}_1$ and $\mathcal{D}_2$ that are either $\{\mathbb{F}\}$ or do not contain the full team, such that $\phi_1\vdashv\Theta^{\prime}_{\mathcal{D}_1}$ and $\phi_2\vdashv\Theta^{\prime}_{\mathcal{D}_2}$.  Hence by soundness $\Theta^{\prime}_{\mathcal{D}_1} \models\Theta^{\prime}_{\mathcal{D}_2}$, and it suffices to derive $\Theta^{\prime}_{\mathcal{D}_1} \vdash\Theta^{\prime}_{\mathcal{D}_2}$

If $\mathcal{D}_1=\{\mathbb{F}\}$, then we derive $\Theta^{\prime}_{\mathcal{D}_1}= \bullet \vdash \Theta^{\prime}_{\mathcal{D}_2}$ by $\bullet$E, and if $\mathcal{D}_2=\{\mathbb{F}\}$, then also $\mathcal{D}_1=\{\mathbb{F}\}$ hence the result is immediate. 

 We can now assume that  $\mathcal{D}_1$ and $\mathcal{D}_2$ do not contain $\mathbb{F}$. By \Cref{lemma wdc} \cref{lemma wdc ii}, we have that for each  $T\in\mathcal{D}_1$ there is some $S\in\mathcal{D}_2$ such that $S\supseteq T$. 
Now $\theta^{\prime}_T\vdash \theta^{\prime}_T\lor\theta^{\prime}_{S\setminus T}=\theta^{\prime}_S$ by $\lor$I which we can use since $\theta^{\prime}_{S\setminus T}$ is $\bullet$-free. We then apply $\glor$I to obtain the derivation $\theta^{\prime}_S\vdash\Theta^{\prime}_{\mathcal{D}_2}$. We conclude by $\glor$E that $\Theta^{\prime}_{\mathcal{D}_1}\vdash\Theta^{\prime}_{\mathcal{D}_2}$, hence $\gamma\vdash\phi$. 
 \end{proof}




 \subsection{Downward closed logic}

 We introduce the logic $\mathcal{L}_d$ to complete the dual picture of the (quasi) upward and downward closed logics we consider in this section. We briefly discuss the expressive completeness and complete proof system for this logic, since they essentially follow from the results for $\PLglor$ in \cite{Yang16}. Recall that $\PLglor$ has the syntax $\phi::= \bot\mid p \mid \neg p \mid (\phi\land\phi) \mid (\phi\lor\phi)\mid (\phi\glor\phi)$, and that the syntax of the downward closed logic $\mathcal{L}_d$ is 
\[\phi::= \bot \mid \mathsf{p}\subseteq \mathsf{x} \mid (\phi\land\phi) \mid (\phi\lor\phi)\mid (\phi\glor\phi).\]

We can use the same semantic arguments as for $\PLglor$ in \cite{Yang16} due to the following equivalences $q\subseteq\top\equiv q$, and $q\subseteq\bot\equiv \neg q$. 

Let $\mathsf{p}$ contain all propositional symbols from $\mathbb{P}$. Since $\mathsf{p}\subseteq\mathsf{x}\equiv\mathsf{p}^\mathsf{x}$, we directly translate the normal form for $\PLglor$ in \cite{Yang16} to obtain 
$$\Theta^{*}_{\mathcal{C}}:=\Glor_{T \in \mathcal{C}} \bigvee_{v\in T} \mathsf{p}\subseteq \mathsf{x}^v,$$

which captures the downward closure of a nonempty team property $\mathcal{C}$. Thus, we omit the proof and simply state the expressive completeness result for $\mathcal{L}_d$. 

\begin{theorem}
    $\mathcal{L}_{d}$ is expressively complete for all nonempty downward closed team properties $\mathcal{C}$.
\end{theorem}

The natural deduction proof system is also similar to the one for $\PLglor$ in \cite{Yang16}, and we define the proof system for $\mathcal{L}_d$ in \Cref{rules:dc_logic}.

\begin{table}[ht]\centering \small
  \renewcommand*{\arraystretch}{3.8}
\begin{tabular*}{\linewidth}{@{\extracolsep{\fill}}|ccc|}
\hline

 \AxiomC{}
\RightLabel{$\top$I}
\UnaryInfC{$\sempty\subseteq\sempty$}
\DisplayProof
& 
\AxiomC{}
\noLine
\UnaryInfC{$q\subseteq\top$}
\AxiomC{}
\noLine
\UnaryInfC{$q\subseteq\bot$}
\RightLabel{$\bot$E }
\BinaryInfC{$\psi$}
\DisplayProof
&
\AxiomC{}
\noLine
\UnaryInfC{$\psi\lor\bot$}
\RightLabel{$\bot\lor$E}
\UnaryInfC{$\psi$}
\DisplayProof
\\[2ex] 

 \AxiomC{}
     \noLine  
     \UnaryInfC{$\mathsf{p}q\subseteq\mathsf{x}y$}
     \RightLabel{$\subseteq$Proj}
\UnaryInfC{$\mathsf{p}\subseteq\mathsf{x}$}
\DisplayProof
&

 \AxiomC{}
     \noLine  
     \UnaryInfC{$\mathsf{u}\mathsf{v}\mathsf{w} \subseteq \mathsf{x}\mathsf{y}\mathsf{z}$}
     \RightLabel{$\subseteq$Perm (1)}
\UnaryInfC{$\mathsf{u}\mathsf{w}\mathsf{v}\subseteq\mathsf{x}\mathsf{z}\mathsf{y}$} 
 \DisplayProof
& 
 \AxiomC{}
     \noLine  
     \UnaryInfC{$\mathsf{p}\subseteq\mathsf{x}$}
     \RightLabel{$\subseteq$Ext}
\UnaryInfC{$\mathsf{p}q\subseteq\mathsf{x}\top\lor\mathsf{p}q\subseteq\mathsf{x}\bot$}

\DisplayProof
 \\ [2ex]
 \hline

\multicolumn{3}{|c|}{
The rules $\land$I, $\land$E, $\glor$I and $\glor$E from \Cref{rules:wuc_logic}.}
 \\

\AxiomC{}
\noLine
\UnaryInfC{$\phi$}
\RightLabel{$\lor$I}
\UnaryInfC{$\phi\lor\psi$}
\DisplayProof
& 
\AxiomC{}
\noLine
\UnaryInfC{$\phi$}
\RightLabel{$\lor$I }
\UnaryInfC{$\psi\lor\phi$}
\DisplayProof
& 
\AxiomC{}
\noLine
\UnaryInfC{$\phi \lor \psi$}
\noLine
\AxiomC{[$\phi$]}
\noLine
\UnaryInfC{\vdots}
\noLine
\UnaryInfC{$\chi$}
\AxiomC{[$\psi$]}
\noLine
\UnaryInfC{\vdots}
\noLine
\UnaryInfC{$\chi$}
\RightLabel{$\lor$E (2)}
\TrinaryInfC{$\chi$}
\DisplayProof\\ 

 \multicolumn{3}{|c|}{
The rules $\lor$Com, $\lor$Mon and $\lor\glor$Distr from  \Cref{rules:dc_logic}.}
 \\ 
 
\hline

\multicolumn{3}{|c|}{(1) Provided that $|y|=|v|$ and $|z|=|w|$. (2) $\chi$ is $\glor$-free.}
 \\ 
\hline

\end{tabular*}
 \caption{System for $\mathcal{L}_d$.}
\label{rules:dc_logic}
\end{table}  

 The rules for the connectives are identical to the system for $\PLglor$ in \cite{Yang16}, except for the rule deriving $\phi\lor(\psi\lor\chi)\vdash (\phi\lor\psi)\lor\chi$ which can be omitted from both systems as discussed when introducing the system for $\mathcal{L}_{qd}$. Our rule $\bot\lor$E also appears in \cite{Yang16}, and the rule $\bot$E is a direct translation of the corresponding one in \cite{Yang16}. By $\top$I and $\subseteq$Ext we can derive the law of excluded middle, which with our notation is $q\subseteq\top\lor q\subseteq\bot$. We add two novel rules to the system: $\subseteq$Proj and $\subseteq$Perm. This is simply to be able to eliminate the hidden conjuncts in the dual primitive inclusion atom, for instance: $qr\subseteq \bot\top\vdash rq\subseteq\top\bot\vdash r\subseteq\top$, which in $\PLglor$ translates to $\neg q\land r\vdash \neg q$, derivable by $\land$E. 

 Based on this discussion, we simply state the soundness and completeness theorem of the system. 
 
\begin{theorem}
  The rules in \Cref{rules:dc_logic} form a sound and complete proof system for $\mathcal{L}_{d}$, i.e., for a set of $\mathcal{L}_{d}$-formulas $\Gamma\cup\{\phi\}$, if $\Gamma\models\phi$ then $\Gamma\vdash \phi$.
\end{theorem}

\subsection{Symmetry and might modalities}





We summarize the expressivity results for the logics by illustrating the syntactical duality of the normal forms between the (quasi) upward and downward closed settings.  

Let $\mathcal{C}$ be any nonempty team property. In each of the four settings we have considered, we define the closure of $\mathcal{C}$ as follows: 

\begin{enumerate}[label=\textbf{-}]

    \item The quasi upward closure of $\mathcal{C}$ is $qu(\mathcal{C}):=\{\emptyset\}\cup \{T\mid T\supseteq S \text{ for some nonempty } S\in\mathcal{C}\}$.

    \item The upward closure of $\mathcal{C}$ is $u(\mathcal{C}):=\{T\mid T\supseteq S \text{ for some } S\in\mathcal{C}\}$.

    \item The quasi downward closure of $\mathcal{C}$ is $qd(\mathcal{C}):=\{\mathbb{F}\}\cup \{T\mid T\subseteq S \text{ for some non full } S\in\mathcal{C}\}$.

    \item The downward closure of $\mathcal{C}$ is $d(\mathcal{C}):=\{T\mid T\subseteq S \text{ for some } S\in\mathcal{C}\}$.

\end{enumerate}

Using the normal forms for the four logics, we conclude that it is possible to capture the (quasi) upward/downward closure of any nonempty team property $\mathcal{C}$. We capture $\{\emptyset\}$ in  $\mathcal{L}_{qu}$ by $\bot$, and $\{\mathbb
F\}$ in  $\mathcal{L}_{qd}$ by $\bullet$. For a nonempty team property $\mathcal{C}$ that is not $\{\emptyset\}$ or $\{\mathbb{F}\}$, we illustrate the syntactical duality in the normal forms of the four logics with which we capture its closure. In particular, in place of the conjunctions found in the normal forms in the (quasi) upward closed settings, there are split-disjunctions in the (quasi) downward closed settings. Furthermore, the variants of the (nonempty) primitive inclusion atoms become (full) dual primitive inclusion atoms in the (quasi) downward closed settings. 

\bigbreak

\fbox{
\begin{minipage}{.5\textwidth}

Quasi upward closed setting:     $$\lVert\Glor_{T\in\mathcal{C}\setminus \{\emptyset\}}\bigwedge_{v\in T} \mathsf{x}^v\subseteq \mathsf{p}\rVert=qu(\mathcal{C}).$$

Upward closed setting:      $$\lVert\Glor_{T\in\mathcal{C}}\bigwedge_{v\in T} \mathsf{x}^v\subseteqq \mathsf{p}\rVert=u(\mathcal{C}).$$

\end{minipage}

\begin{minipage}{.4\textwidth}

Quasi downward closed setting: 
   $$\lVert\Glor_{T \in \mathcal{C}\setminus\{\mathbb{F}\}} \bigvee_{v\in T} \mathsf{p}\fullsubseteq \mathsf{x}^v\rVert=qd(\mathcal{C}).$$

       Downward closed setting: 
$$\lVert\Glor_{T \in \mathcal{C}} \bigvee_{v\in T} \mathsf{p}\subseteq \mathsf{x}^v\rVert=d(\mathcal{C}).$$
\end{minipage}
}

\bigbreak

We end by making the connection between (nonempty) primitive inclusion atoms and (nonempty) might modalities from the literature \cite{Veltman,HellaStumpf15, AHY}, \cite{anttila2025convex}. 
Recall the semantics of three might modalities:  
 \begin{align*}
 T\models\singlemight\phi\quad \text{iff}\quad& T=\emptyset, \text{ or there is } v\in T \text{ such that } \{v\}\models\phi.\\
T\models\might\phi\quad \text{iff}\quad& T=\emptyset, \text{ or there is a nonempty } S\subseteq T \text{ such that } S\models\phi.\\
T\models\blackdiamond\phi\quad \text{iff}\quad&\text{there is a nonempty } S\subseteq T \text{ such that } S\models\phi.
\end{align*}
%

We list some semantic connections between the (nonempty) primitive inclusion atoms and the might modalities.
\[\text{$ \top\subseteq p\equiv \might p$,\quad 
$\mathsf{x}^v\subseteq \mathsf{p}\equiv \might (\mathsf{p}^{v})$,\quad
$\top\subseteqq p\equiv \blackdiamond p$, \quad and \quad 
$\mathsf{x}^v\subseteqq \mathsf{p}\equiv \blackdiamond (\mathsf{p}^{v})$,}\]

It is easy to see that for classical formulas $\alpha$, $\singlemight\alpha\equiv\might\alpha$, hence the equivalences above that hold for $\might$ also hold for $\singlemight$.


We can also consider the dual to the correspondence $\top\subseteq p$ and `might $p$', as the one between $p\subseteq \top$ and `must $p$', since $T\models p\subseteq\top$ means that $p$ \emph{must} be true in the whole team. 

Furthermore, we can consider teams as information states, where asserting \emph{whether} $p$ amounts to $p$ having the same truth value in the whole team. `Might $p$' thus describes the possibility that after refining the information state, $p$ might be asserted as true. This view also motivates the full team property as a dual to the empty team property. If a team represents an information state, the full team represents all possibilities, and thus no information. Conversely, for a team seen as a set of data points, the empty team represents no data.



\section{Conclusion and future work}

Let us end with a summary. Firstly, we introduce logics expressively complete for (quasi) downward and (quasi) upward closed properties. In particular, the variants of the primitive inclusion atoms used in the (quasi) upward closed setting have equivalent formulas using variants of the might modality. With this perspective, we can interpret the variants of the dual primitive inclusion atoms used in the downward closed setting as must modalities. Moreover, a syntactic duality between the (quasi) downward and (quasi) upward closed settings is evident in the normal forms of the logics. Lastly, we defined sound and complete natural deduction systems for each logic. 

Let us identify some directions of future work. The first being to study logics that essentially combine a (quasi) downward closed logic with a (quasi) upward closed one.

\begin{enumerate}[label=\textbf{-}]
    \item We extend $PL(\subseteq)$ with the global disjunction to obtain
$\mathcal{L}_{\emptyset}$, which
is expressively complete for all team properties with the empty team.
    \[\phi::= \bot \mid p \mid \neg p \mid \mathsf{x}\subseteq \mathsf{p} \mid (\phi\land\phi)\mid (\phi\lor\phi)\mid (\phi\glor\phi).\]

It is easy to see that $\mathcal{L}_{\emptyset}$ has the normal forms $\bot$ and: 
$$\Glor_{t\in\mathcal{C}\setminus \{\emptyset\}} (\bigvee_{v\in t} \chi_{v}\land \bigwedge_{v\in t} \mathsf{x}^v\subseteq \mathsf{p}).$$

\item $\mathcal{L}_{\mathsf{F}}$ is expressively complete for all team properties with the full team.
    \[\phi::= \bullet \mid \top \mid \mathsf{p}\fullsubseteq \mathsf{x} \mid \mathsf{x}\subseteq \mathsf{p} \mid (\phi\land\phi)\mid (\phi\lor\phi)\mid (\phi\glor\phi).\]

Now $\mathcal{L}_{\mathsf{F}}$ has the normal forms $\bullet$ and: 
$$\Glor_{t\in\mathcal{C}\setminus \{\mathsf{F}\}} (\bigvee_{v\in t} \mathsf{p}\fullsubseteq \mathsf{x}^v\land \bigwedge_{v\in t} \mathsf{x}^v\subseteq \mathsf{p}).$$
\end{enumerate}

As future work remains the axiomatizations of these two logics. 

For the (quasi) upward closed logics we covered in this paper, one can study them more closely by considering connections to natural language, complexity questions,  sequent calculus proof systems, and their modal variants.

\section*{}
\bibliographystyle{plainurl}
\bibliography{mybib}

\end{document}